\documentclass[onefignum,onetabnum]{siamart190516}

\usepackage{lipsum}
\usepackage{amsfonts}
\usepackage{graphicx}
\usepackage{epstopdf}
\usepackage{algorithmic}
\ifpdf
  \DeclareGraphicsExtensions{.eps,.pdf,.png,.jpg}
\else
  \DeclareGraphicsExtensions{.eps}
\fi
\usepackage{float}
\usepackage{graphicx}
\usepackage{booktabs} 
\usepackage{bm}
\usepackage{setspace}
\usepackage{enumitem}
\usepackage{graphicx}
\usepackage[mathscr]{euscript}
\usepackage{cite}

\newsiamremark{remark}{Remark}
\newsiamremark{hypothesis}{Hypothesis}
\crefname{hypothesis}{Hypothesis}{Hypotheses}
\newsiamthm{claim}{Claim}

\newsiamthm{prop}{Proposition}
\newsiamthm{example}{Example}

\renewcommand{\vec}[1] {\ensuremath{\boldsymbol{#1}}}

\headers{HDSA for ill-posed inverse problems}{J. Hart and B. van Bloemen Waanders}

\title{Enabling hyper-differential sensitivity analysis for ill-posed inverse problems \thanks{Submitted to the editors DATE.
\funding{This paper describes objective technical
results and analysis. Any subjective views or opinions that might be
expressed in the paper do not necessarily represent the views of the
U.S. Department of Energy or the United States Government. Sandia
National Laboratories is a multimission laboratory managed and
operated by National Technology and Engineering Solutions of Sandia
LLC, a wholly owned subsidiary of Honeywell International, Inc., for
the U.S. Department of Energy's National Nuclear Security
Administration under contract DE-NA-0003525. SAND2021-6604 O
This work was supported by the US Department of Energy, Office of Advanced Scientific Computing Research, Field Work Proposal 20-023231.}}}

\author{Joseph Hart\thanks{Optimization and Uncertainty Quantification, Sandia National Laboratories, P.O. Box 5800, Albuquerque, NM 87123-1320
  (\email{joshart@sandia.gov}).}
\and Bart van Bloemen Waanders \thanks{Optimization and Uncertainty Quantification, Sandia National Laboratories, P.O. Box 5800, Albuquerque, NM 87123-1320
  (\email{bartv@sandia.gov}).} }

\usepackage{amsopn}

\ifpdf
\hypersetup{
  pdftitle={Enabling hyper-differential sensitivity analysis for ill-posed inverse problems},
  pdfauthor={Joseph Hart and Bart van Bloemen Waanders}
}
\fi

\begin{document}

\maketitle

\begin{abstract}
Inverse problems constrained by partial differential equations (PDEs) play a critical role in model development and calibration. In many applications, there are multiple uncertain parameters in a model that must be estimated. However, high dimensionality of the parameters and computational complexity of the PDE solves make such problems challenging. A common approach is to reduce the dimension by fixing some parameters (which we will call auxiliary parameters) to a best estimate and use techniques from PDE-constrained optimization to estimate the other parameters. In this article, hyper-differential sensitivity analysis (HDSA) is used to assess the sensitivity of the solution of the PDE-constrained optimization problem to changes in the auxiliary parameters. Foundational assumptions for HDSA require satisfaction of the optimality conditions which are not always practically feasible as a result of ill-posedness in the inverse problem. We introduce novel theoretical and computational approaches to justify and enable HDSA for ill-posed inverse problems by projecting the sensitivities on likelihood informed subspaces and  defining a posteriori updates. Our proposed framework is demonstrated on a nonlinear multi-physics inverse problem motivated by estimation of spatially heterogenous material properties in the presence of spatially distributed parametric modeling uncertainties. 
\end{abstract}



\section{Introduction}
Inverse problems arise in scientific and engineering applications when quantities cannot be directly observed but rather are estimated using observations of  related quantities.  In the case of inverse problems constrained by partial differential equations (PDEs), the state variables (solution of the PDE) are observed at (potentially sparse) locations in space and/or instances in time, and from these observations one seeks to infer parameters in the model. There is a wealth of literature developing theory and computational methods to solve these challenging yet scientifically critical problems.  We refer the interested reader to a small sampling of the literature and the references therein \cite{ghattas_infinite_dim_bayes_1,ghattas_infinite_dim_bayes_2,stuart_inv_prob,Biegler_11}. 

Among the many challenges faced in PDE-constrained inverse problems, data sparsity, high dimensionality, and computational complexity are at the forefront. It is common for PDE-based models of physical processes to have many uncertain parameters, some of which are spatially and/or temporally distributed. Ideally, all uncertain parameters are inverted for simultaneously to characterize dependences between them; however, this is challenging because the high dimensionality of the parameter space induces intractable ill-conditioning. A simplification employed in practice is to fix some of the uncertain parameters to a best estimate which reduces the dimension and improves numerical conditioning. Although pragmatic, this fail to adequately address dependence between the parameters.  

In this article we focus on deterministic inverse problems and consider the sensitivity of the solution to the inverse problem due to perturbations in parameters. 
Although these fixed parameters are sometimes called nuisance parameters in the inverse problems literature, in this article we adopt the term auxiliary parameters for greater generality. For problems with separable structure, variable projection may be used to write the parameter of interest as a function of the auxiliary parameters \cite{varpro_nuisance_parameters}. The error introduced by fixing auxiliary parameters to a nominal estimate may be represented in a Bayesian approximation error approach to incorporate their uncertainty in the inversion \cite{bayes_approx_error_petra,stat_inv_prob_book}. Alternatively, this article uses hyper-differential sensitivity analysis (HDSA) to assess the influence of the auxiliary parameters on the solution of a PDE-constrained optimization problem. By differentiating through optimality conditions, and leveraging adjoint-based derivative calculations along with randomized and matrix free linear algebra, HDSA has been shown to provide a computationally scalable approach to analyze uncertainty in large-scale optimization \cite{HDSA,sunseri_hdsa,saibaba_gsvd}. The resulting sensitivities provide critical information to characterize dependencies between the inversion parameters and fixed auxiliary parameters. This characterization of dependences supports model development, calibration, and data acquisition strategies~\cite{sunseri_hdsa,hdsa_control,Sunseri_2022}.

Large-scale PDE-constrained inverse problems, even with auxiliary parameters fixed, are frequently ill-posed as a result of the parameter dimensionality and data sparsity. This ill-posedness creates theoretical and computational challenges for the HDSA framework \cite{HDSA} since it is build on the foundation of post-optimality sensitivity analysis \cite{Griesse_part_1,shapiro_SIAM_review}. In particular, satisfying first and second order optimality conditions is required for HDSA.  However, such satisfaction may not be easily attained for many inverse problems where prior information is limited, nonlinearities are present, and limited computational resources are available. Failure to satisfy these assumptions prohibits the use of HDSA. This article introduces novel theoretical and computational approaches to enable HDSA for ill-posed PDE-constrained inverse problems.  Our contributions include:
\begin{enumerate}
\item[$\bullet$] defining HDSA on the likelihood informed subspace~\cite{cui_2014}, the low dimensional subspace informed by data, and demonstrating its advantages to combat ill-conditioning due to ill-posedness and data sparsity,
\item[$\bullet$] developing first and second order a posteriori updates to overcome the theoretical limitations of HDSA when optimality conditions are not satisfied, 
\item[$\bullet$] providing an algorithmic and computational framework to apply HDSA to ill-posed inverse problems and demonstrating its effectiveness on a nonlinear multi-physics application.
\end{enumerate}

The article is organized as follows. Section~\ref{sec:background} reviews PDE-constrained inverse problems and HDSA. We introduce HDSA on the likelihood informed subspace in Section~\ref{sec:likelihood_informed} to facilitate a more useful interpretation of the sensitivities for ill-posed inverse problems. A posteriori updates are developed in Section~\ref{sec:enabling_hdsa} to overcome the practical challenges associated slow optimizer convergence which limit the use of HDSA. An algorithmic overview is presented in Section~\ref{sec:alg_overview}. We demonstrate our approach on a multi-physics subsurface flow application in Section~\ref{sec:numerics} with uncertainty in a permeability field, source, boundary conditions, and diffusion coefficient. Section~\ref{sec:conclusion} concludes the article with a perspective on how these developments relate to the Bayesian formulation of the inverse problem and opportunities for future work along these lines.

\section{Background} \label{sec:background}
We consider inverse problems constrained by partial differential equations (PDEs). In many applications, the PDE depends on multiple uncertain parameters; however, inverting for all parameters simultaneously is frequently intractable. Rather, the common scenario in practice is to fix some parameters to a best estimate and only invert for the parameters of greatest interest. To this end, let $\vec{z} \in \mathbb R^m$ be the parameters of interest being inverted for and $\vec{\theta} \in \mathbb R^n$ denote parameters being fixed to a nominal estimate. We will refer to $\vec{\theta}$ as the auxiliary parameters since they are not of primary interest to estimate. Though results in this article are applicable for low dimensional problems, we are generally interested in cases where $\vec{z}$ and/or $\vec{\theta}$ correspond to the discretization of infinite dimensional parameters and hence $m$ and/or $n$ are high dimensional.

Let $F:\mathbb R^m \times \mathbb R^n \to \mathbb R^d$ denote the parameter-to-observable map. Evaluating $F(\vec{z},\vec{\theta})$ involves solving the PDE and then applying an observation operator, i.e. a map from the PDE solution space to the space of observable data. Evaluating $F$ is computationally costly due to the PDE solve. 

 Given data $\vec{d} \in \mathbb R^d$ and a nominal estimate $\overline{\vec{\theta}}$ for the auxiliary parameters, the inverse problem is to determine $\vec{z}$ such that 
$$F(\vec{z},\overline{\vec{\theta}}) \approx \vec{d}.$$
Since observed data is typically noisy and models are imperfect, we do not seek $F(\vec{z},\overline{\vec{\theta}}) = \vec{d}$, but rather that $F(\vec{z},\overline{\vec{\theta}}) - \vec{d}$ is on the same order of magnitude as the noise in the data. Inverse problems are frequently ill-posed in the sense that there are many $\vec{z}$'s, which may be far from one another, such that $F(\vec{z},\overline{\vec{\theta}}) - \vec{d}$ has the same magnitude as the data's noise. This causes considerable challenges when solving the inverse problem.

\subsection{Optimization problem}
We focus on the optimization problem
\begin{align}\label{inv_prob_RS}
& \min\limits_{\vec{z} \in \mathbb R^m} J(\vec{z};\overline{\vec{\theta}}):=M(\vec{z},\overline{\vec{\theta}}) + R(\vec{z})
\end{align}
where 
\begin{eqnarray*}
M(\vec{z},\overline{\vec{\theta}}) = \frac{1}{2} (F(\vec{z},\overline{\vec{\theta}})-\vec{d})^T W (F(\vec{z},\overline{\vec{\theta}})-\vec{d}) 
\end{eqnarray*}
is the data misfit weighted by a symmetric positive definite matrix $W$ and $R(\vec{z})$ is a regularization function.  Since there are typically many $\vec{z}$'s for which $M(\vec{z},\overline{\vec{\theta}})$ is small, the choice of the regularization function $R$ is critical to impose prior knowledge on $\vec{z}$. Though not a focus for this article, the choice of regularization is linked to the prior in a Bayesian inverse problem \cite{isaac_2015,ghattas_infinite_dim_bayes_1,ghattas_infinite_dim_bayes_2}.

Computing the solution of~\eqref{inv_prob_RS} for large-scale inverse problems is computationally challenging since it requires many PDE solves \cite{ghattas_infinite_dim_bayes_1,ghattas_infinite_dim_bayes_2,Biegler_11}. However, techniques from PDE-constrained optimization may be leveraged to facilitate computational efficiency. Techniques include finite element discretization, matrix free linear algebra, adjoint-based derivative computation, and parallel computing. The reader is referred to \cite{Vogel_99, Archer_01,Haber_01,Vogel_02,Biegler_03,Biros_05,Laird_05,Hintermuller_05,Hazra_06,Biegler_07,Borzi_07,Hinze_09,Biegler_11,frontier_in_pdeco} for a sampling of the PDE-constrained optimization literature. 

Solving~\eqref{inv_prob_RS} with auxiliary parameter $\overline{\vec{\theta}}$ provides an estimate of $\vec{z}$.  However, this estimate is not accurate due to uncertainty in $\vec{\theta}$. Ideally, we would solve~\eqref{inv_prob_RS} for many different $\vec{\theta}$'s but because this is computationally prohibitive, we use HDSA to efficiently analyze the dependence of our estimated $\vec{z}$ on the auxiliary parameter $\vec{\theta}$.

\subsection{Hyper-differential sensitivity analysis}
\label{ssec:hdsa_intro}
Through a combination of tools from post-optimality sensitivity analysis, PDE-constrained optimization, and numerical linear algebra, HDSA has provided unique and valuable insights into the dependence of PDE-constrained inverse problems on high dimensional auxiliary parameters \cite{HDSA,sunseri_hdsa,saibaba_gsvd}.
This subsection provides essential background on HDSA. To facilitate our analysis, assume that the objective function in \eqref{inv_prob_RS}, $J:\mathbb R^m \times \mathbb R^n \to \mathbb R$, is twice continuously differentiable with respect to $(\vec{z},\vec{\theta})$. Let $\vec{z}^\star$ be a local minimum of~\eqref{inv_prob_RS} when the auxiliary parameters are fixed to $\overline{\vec{\theta}} \in \mathbb R^n$. A fundamental assumption is that $\vec{z}^\star$ satisfies the well known first and second order optimality conditions:
\begin{enumerate}[label=(A\textnormal{\arabic*})]
\item $\nabla_{\vec{z}} J(\vec{z}^\star;\overline{\vec{\theta}})=0$, \label{A1} 
\item $\nabla_{\vec{z},\vec{z}} J(\vec{z}^\star; \overline{\vec{\theta}})$ is positive definite, \label{A2}
\end{enumerate}
where $\nabla_{\vec{z}} J$ and $\nabla_{\vec{z},\vec{z}} J$ denote the gradient and Hessian of $J$ with respect to $\vec{z}$, respectively. 

Thanks to our assumption on differentiability of $J$, we may apply the Implicit Function Theorem~\cite{krantz_2013} to the first order optimality condition $\nabla_{\vec{z}} J(\vec{z}^\star;\overline{\vec{\theta}})=0$. This gives the existence of a continuously differentiable operator $\mathcal G:\mathcal N(\overline{\vec{\theta}}) \to \mathbb R^m$, defined on a neighborhood $\mathcal N(\overline{\vec{\theta}})$, such that $\nabla_{\vec{z}} J(\mathcal G(\vec{\theta}); \vec{\theta})=0$ for all $\vec{\theta} \in \mathcal N(\overline{\vec{\theta}})$. Assuming that $\nabla_{\vec{z},\vec{z}} J(\mathcal G(\vec{\theta}); \vec{\theta})$ is positive definite for all $\vec{\theta} \in \mathcal N(\overline{\vec{\theta}})$, we may interpret $\mathcal G$ as a map from auxiliary parameters $\vec{\theta}$ to the solution of the optimization problem~\eqref{inv_prob_RS}. 

Furthermore, it follows from the Implicit Function Theorem that the Jacobian of $\mathcal G$, evaluated at $\overline{\vec{\theta}}$, is given by
\begin{eqnarray}
\label{opt_sol_sensitivity}
\mathcal G'(\overline{\vec{\theta}}) = -\mathcal H^{-1} \mathcal B,
\end{eqnarray}
where $\mathcal B=\nabla_{\vec{z},\vec{\theta}} J(\vec{z}^\star,\overline{\vec{\theta}})$ denotes the Jacobian of $\nabla_{\vec{z}} J$ with respect to $\vec{\theta}$, and $\mathcal H=\nabla_{\vec{z},\vec{z}} J(\vec{z}^\star,\overline{\vec{\theta}})$ denotes the Hessian of $J$ with respect to $\vec{z}$, each evaluated at $\vec{z}=\vec{z}^\star$ and $\vec{\theta}=\overline{\vec{\theta}}$. Hence we may interpret $\mathcal G'(\overline{\vec{\theta}}) \vec{\theta}_0$ as the change in our estimate of $\vec{z}$ when $\overline{\vec{\theta}}$ is perturbed in the direction $\vec{\theta}_0$.

Note that HDSA may be formally developed for infinite dimensional problems in a full space optimization framework. The reader is directed to \cite{HDSA} for additional details. Also note that $\mathcal G'(\overline{\vec{\theta}})$ is a local sensitivity around $\overline{\vec{\theta}}$. Global sensitivities may be considered, see~\cite{HDSA}, but are beyond the scope of this article.

\section{Projection on the likelihood informed subspace}
\label{sec:likelihood_informed}
For ill-posed inverse problems, it is common that $\mathcal H$ has small eigenvalues, or equivalently, $\mathcal H^{-1}$ has large eigenvalues. If the span of the eigenvectors corresponding to these eigenvalues intersects the range of $\mathcal B$, then $\mathcal G'(\overline{\vec{\theta}})$ will be large in those directions. However, changes in the optimal solution along directions which are poorly informed by the data do not provide physical insight into the relationship between $\vec{z}$ and $\vec{\theta}$. Rather, such uninformed directions can be analyzed in the posterior distribution of $\vec{z}$. We focus on sensitivities which provide physical insight informed by the data. This motives us to define a projector onto the likelihood informed subspace (LIS)~\cite{cui_2014}. The LIS is defined by $r$ largest positive eigenvectors of the generalized eigenvalue problem 
\begin{eqnarray}
\label{eqn:likelihood_informed_gen_eig}
\mathcal H_M \vec{v}_j = \lambda_j \mathcal H_R \vec{v}_j, \qquad j=1,2,\dots,r,
\end{eqnarray} 
where
\begin{eqnarray*}
 \mathcal H_M=\nabla_{\vec{z},\vec{z}}^2 M(\vec{z}^\star,\overline{\vec{\theta}}) \qquad \text{and} \qquad \mathcal H_R = \nabla_{\vec{z},\vec{z}}^2 R(\vec{z}^\star)
 \end{eqnarray*} 
  are the misfit Hessian and regularization Hessian, respectively. These dominant eigenvectors coincide with the directions in parameter space which are informed by the misfit more than the regularization. This may be observed by multiplying \eqref{eqn:likelihood_informed_gen_eig} by $\vec{v}_j^T$ and solving for the Rayleigh quotient
\begin{eqnarray}
\label{eqn:eig_interpretation}
\lambda_j = \frac{\vec{v}_j^T\mathcal H_M \vec{v}_j}{\vec{v}_j^T \mathcal H_R \vec{v}_j}, \qquad j=1,2,\dots,r.
\end{eqnarray} 
The eigenvalues measure the ratio of contributions from the misfit and regularization in the directions of the eigenvectors. A ``global LIS" may be defined by computing the expected misfit Hessian with respect to $\vec{z}$. However, since HDSA is local about $\vec{z}^\star$ our LIS definition provides a useful subspace for what follows.
   
We define the projector 
\begin{eqnarray}
\label{eq:projector}
\mathcal P = \mathcal V \mathcal V^T \mathcal H_R
\end{eqnarray}
where the columns of $\mathcal V=[\vec{v}_1,\vec{v}_2,\dots,\vec{v}_r]$ are the eigenvectors of \eqref{eqn:likelihood_informed_gen_eig} normalized so that $\vec{v}_j^T \mathcal H_R \vec{v}_j=1$, $j=1,2,\dots,r$. This corresponds to computing hyper-differential sensitivities in the directions which are informed more by the misfit rather than the regularization. The truncation rank $r$ is chosen by the user by leveraging the interpretation of the eigenvalues as the ratio of misfit and regularization. For ill-posed inverse problems, $r$ is typically small as a result of data sparsity and dissipative physics. 
 
To associate sensitivity with individual parameters we define the hyper-differential sensitivity indices 
\begin{eqnarray}
\label{sensitivity_indices}
S_i = \vert \vert \mathcal P \mathcal G'(\overline{\vec{\theta}}) \vec{e}_i \vert \vert_{W_{\vec{z}}} = \vert \vert \mathcal P \mathcal H^{-1} \mathcal B \vec{e}_i \vert \vert_{W_{\vec{z}}} \qquad i=1,2,\dots,n,
\end{eqnarray}
where $\vec{e}_i \in \mathbb R^n$ is the $i^{th}$ canonical basis vector which has 1 in its $i^{th}$ entry and 0 in all others. The norm $\vert \vert \cdot \vert \vert_{W_{\vec{z}}}$ is defined by a symmetric positive definite weighting matrix $W_{\vec{z}} \in \mathbb R^{m \times m}$ for generality. This is important if $\vec{z}$ corresponds to a discretization of a function as $W_{\vec{z}}$ encodes the inner products from the infinite dimensional space. We interpret $S_i$ as the change in the LIS projection of the solution if the $i^{th}$ auxiliary parameter is perturbed.
 
 Computing $\mathcal P \mathcal H^{-1}$ as in~\eqref{sensitivity_indices} corresponds to spectral regularization of the Hessian to eliminate ill-conditioning due to the uninformed directions. Theorem~\ref{thm:sen_indices_likelihood_informed} provides an expression for the sensitivity indices which only depends on the $r$ leading generalized eigenvalues and eigenvectors in \eqref{eqn:likelihood_informed_gen_eig}. This is computationally adventageous for ill-posed inverse problems since the leading generalized eigenvalues and eigenvectors of \eqref{eqn:likelihood_informed_gen_eig} can be computed efficiently. The computational benefit is significant for large-scale applications where inverting $\mathcal H$ without preconditioning would take many conjugate gradient iterations to converge.
\begin{theorem}
\label{thm:sen_indices_likelihood_informed}
If $\mathcal H$ is positive definite then the LIS-hyper-differential sensitivities $S_i$, $i=1,2,\dots,n$, (where $\mathcal P$ is defined by~\eqref{eq:projector}) are given by
\begin{eqnarray}
S_i = \left\vert \left\vert \mathcal P \mathcal H^{-1} \mathcal B \vec{e}_i \right\vert \right\vert_{W_{\vec{z}}} = \sqrt{\sum\limits_{k=1}^r \sum\limits_{j=1}^r \left( \frac{\vec{v}_j^T \mathcal B \vec{e}_i }{1+\lambda_j} \right) \left( \frac{\vec{v}_k^T \mathcal B \vec{e}_i }{1+\lambda_k} \right) \vec{v}_k^T W_{\vec{z}}\vec{v}_j}. \label{eqn:sen_indices_formula}
\end{eqnarray}
\end{theorem}
For conciseness and clarity of the presentation, proofs for Theorem~\ref{thm:sen_indices_likelihood_informed} and all subsequent theorems are given in the appendix. Theorem~\ref{thm:sen_indices_likelihood_informed} is similar in concept to using a low rank approximation of the prior preconditioned misfit hessian~\cite{ghattas_infinite_dim_bayes_2}. 

\section{Enabling HDSA for ill-posed inverse problems} \label{sec:enabling_hdsa}
For ill-posed inverse problems, it is frequently the case that a user does not solve~\eqref{inv_prob_RS} to optimality, i.e. the solution may fail to satisfy the first and/or second order optimality conditions. This lack of optimality occurs for a variety of interrelated reasons:
\begin{enumerate}
\item[$\bullet$] The objective function $J$ is flat in directions which are not well informed by the data and lacks strong regularization to overcome this uncertainty.
\item[$\bullet$] The Hessian is ill-conditioned since the directions along which the objective function is flat corresponds to small eigenvalues. Hence, numerical optimization schemes exhibit slow convergence around the local minimum as a result of error in second order information (for instance, error in the linear solve of a Newton step because of ill-conditioning).  
\item[$\bullet$] Due to limited prior information, solving to optimality (which is computationally intensive) is unnecessary.
\end{enumerate}
A common observation in ill-posed problems is that the first few iterations of a numerical optimization routine is significantly informed by the data with a corresponding decrease in the objective function $J$, followed by potentially many time consuming iterations that produce small changes in the objective as the optimizers moves in the directions where $J$ is flat. Hence, it is advantageous for a user to terminate the optimization routine prematurely yielding a $\vec{z}^\star$ which fails to satisfy Assumption~\ref{A1}, and possibly Assumption~\ref{A2}. 

These observations pose significant theoretical and numerical challenges when applying HDSA to ill-posed inverse problems. The optimality conditions required to define and interpret the operator $\mathcal G'$ in~\eqref{opt_sol_sensitivity} are not satisfied. This calls into question the theoretical validity of HDSA in this context and begs the question how it may be generalized. In this section we propose ``a posteriori updates" to provide a theoretical foundation needed to justify HDSA when optimality conditions are not satisfied. In particular, we introduce first and second order updates when Assumptions~\ref{A1} and~\ref{A2} are not satisfied, and demonstrate that these updates introduce negligible computational overhead and the resulting inferences are robust to perturbations of the updates.

\subsection{First order a posteriori update}
\label{subsec:first_order_update}
Assume that an optimization routine has been used to solve~\eqref{inv_prob_RS} and was terminated before convergence yielding a solution $\vec{z}^\star$ for which $\nabla_{\vec{z}} J(\vec{z}^\star;\overline{\vec{\theta}}) \ne 0$, i.e. it fails to satisfy  Assumption~\ref{A1}. For this subsection we will assume that Assumption~\ref{A2} is satisfied. Subsection~\ref{subsec:second_order_update} considers the case when it is not. 

One may compute the expression for the hyper-differential sensitivity indices~\eqref{sensitivity_indices}; however, it lacks theoretical justification since the derivation of~\eqref{sensitivity_indices} is based on the assumption that $\vec{z}^\star$ satisfies the first order optimality condition, Assumption~\ref{A1}. To facilitate HDSA we propose to ``update" the objective function so that $\vec{z}^\star$ is a local minimum of the updated objective. This corresponds to solving a perturbed inverse problem which is justified if the change in the problem is only associated with directions which are poorly informed by the data. We will modify the regularization and show that the HDSA results are robust since we project onto the likelihood informed subspace. In other words, the lack of theoretical validity arising from failure to meet the first order optimality condition is alleviated  when sub-optimality occurs in uninformed directions.

We propose an additive update to the regularization (replacing $R$ with $R+\tilde{R}$). In other words, the perturbed inverse problem has the same misfit objective and an updated regularization. To facilitate our analysis, we require that the regularization update $\tilde{R}$ be quadratic, convex, and non-negative. To ensure that the perturbed inverse problem is ``close" to the original problem, we seek a minimum norm function $\tilde{R}$ for which $\vec{z}^\star$ is a stationary point of the updated objective $\tilde{J}:=J+\tilde{R}$. We solve
\begin{align}
\label{R_update_gen_opt}
& \min\limits_{\tilde{R} \in Q} \vert \vert \tilde{R} \vert \vert_{L^1(\mu)} \\
\text{s.t.} \ & \nabla_{\vec{z}} \tilde{R}(\vec{z}^\star) = -\nabla_{\vec{z}} J(\vec{z}^\star;\overline{\vec{\theta}}) \nonumber
\end{align}
where
\begin{eqnarray*}
Q = \{\tilde{R}:\mathbb R^m \to \mathbb R \vert \tilde{R} \ge 0, \text{ } \tilde{R} \text{ is quadratic, } \tilde{R} \text{ is convex} \}
\end{eqnarray*}
is the set of candidate updates and $\vert \vert \tilde{R} \vert \vert_{L^1(\mu)}$ is the $L^1$ norm with respect to a Gaussian measure $\mu$ with mean $\vec{z}^\star$ and covariance matrix $\alpha^2 I$. This choice of norm emphasizes the size of $\tilde{R}$ around the solution $\vec{z}^\star$ with weighting from the covariance matrix $\alpha^2 I$ having length scale $\alpha$ in all directions. We assume $\vec{z}$ has been scaled appropriately and define $\alpha$ as one half the distance between the minimum and maximum values in $\vec{z}^\star$. Taking small values of $\alpha$ results in only considering characteristics of $\tilde{R}$ in a small neighborhood of $\vec{z}^\star$ while taking very large values of $\alpha$ results in finding an update $\tilde{R}$ which is very flat (nearly a constant function). Later in the article (Theorem~\ref{thm:vary_alpha}) we will revisit the choice of $\alpha$ and show that our proposed framework is robust with respect to perturbations of $\alpha$. More general measures and norms on $\tilde{R}$ may be considered; however, this $L^1$ norm with a Gaussian measure is preferred because of its clear interpretation and amenability for analysis.

To enable computational efficiency and rigorous analysis, a closed form solution of~\eqref{R_update_gen_opt} is given by Theorem~\ref{thm:optimal_update_1}. The proof (given in the appendix) leverages our choice of the $L^1$ norm and properties of the Gaussian measure (closed form expression for the expectation of quadratic functions) to transform the function space optimization problem into a linear algebra problem which may be analyzed using the spectral decomposition. 
\begin{theorem}
\label{thm:optimal_update_1}
The global minimizer of~\eqref{R_update_gen_opt} is given by
\begin{eqnarray}
\label{opt_R}
\tilde{R}(\vec{z}) = \frac{\alpha}{2} \vert \vert \vec{g} \vert \vert_2 - (\vec{z}-\vec{z}^\star)^T\vec{g}  + \frac{1}{2} (\vec{z}-\vec{z}^\star)^T \frac{1}{\alpha \vert \vert \vec{g} \vert \vert_2 } \vec{g} \vec{g}^T(\vec{z}-\vec{z}^\star),
\end{eqnarray}
where $\vec{g}=\nabla_{\vec{z}} J(\vec{z}^\star;\overline{\vec{\theta}}) \ne 0$ is the gradient of $J$ evaluated at $\vec{z}^\star$ and $\overline{\vec{\theta}}$.
\end{theorem}

The form~\eqref{opt_R} of the optimal update $\tilde{R}$ is unsurprising in that its Hessian is proportional to the outer product of the nonzero gradient with itself, ensuring positive curvature in the gradient direction while having zero curvature in all other directions. The following observations:
\begin{enumerate}
\item[$\bullet$] $\tilde{R}(\vec{z}^\star)=\frac{\alpha}{2} \vert \vert \vec{g} \vert \vert_2$,
\item[$\bullet$] the maximum of $\tilde{R}$ in a size $\alpha$ ball around $\vec{z}^\star$ is $2 \alpha \vert \vert \vec{g} \vert \vert_2$,
\item[$\bullet$] the norm of $\tilde{R}$ is given by $\vert \vert \tilde{R} \vert \vert_{L^1(\mu)} = \alpha \vert \vert \vec{g} \vert \vert_2$,
\end{enumerate}
demonstrate that $\tilde{R}$ is well behaved as a function in the sense that it does not vary rapidly. Furthermore, its global minima are 
$$\vec{z}^\star+\alpha \frac{\vec{g}}{\vert \vert \vec{g} \vert \vert_2}+\vec{u}$$
 where $(\vec{u},\vec{g})_2=0$, hence, it is not overfitting to $\vec{z}^\star$ since the minima of $\tilde{R}$ may be far from $\vec{z}^\star$.

By adding $\tilde{R}$, we define the perturbed optimization problem 
\begin{align} 
\label{inv_prob_RS_pert}
& \min\limits_{\vec{z} \in \mathbb R^m} \tilde{J}(\vec{z};\vec{\theta}):=J(\vec{z};\vec{\theta})+ \tilde{R}(\vec{z}),
\end{align}
for which $\vec{z}^\star$ is a local minima when $\vec{\theta}=\overline{\vec{\theta}}$. Assuming that $\nabla_{\vec{z},\vec{z}} \tilde{J}(\vec{z}^\star;\overline{\vec{\theta}})$ is positive definite (subsection~\ref{subsec:second_order_update} considers when it is not), we may compute and interpret hyper-differential sensitivity indices for the perturbed optimization problem whose solution coincides with the solution of~\eqref{inv_prob_RS_pert}. 

It is possible to compute ``sensitivities" (which are not theoretically justified) using \eqref{sensitivity_indices} or \eqref{eqn:sen_indices_formula} for the original optimization problem~\eqref{inv_prob_RS} despite failures to meet optimality conditions. This raises the question: if sensitivity indices are computed for the original problem ignoring the fact that the gradient norm is nonzero, how much will they differ from the sensitivity indices computed for the perturbed problem~\eqref{inv_prob_RS_pert}, which are theoretically sound? Theorem~\ref{thm:compare_S_and_tilde_S} answers this question.

\begin{theorem}
\label{thm:compare_S_and_tilde_S}
Let $\mathcal H$ and $\tilde{\mathcal H}$ denote the Hessian of $J$ and $\tilde{J}$ with respect to $\vec{z}$, evaluated at $\vec{z}^\star$ and $\overline{\vec{\theta}}$, respectively, and $\mathcal B$ and $\tilde{\mathcal B}$ denote the Jacobian of $\nabla_{\vec{z}} J$ and $\nabla_{\vec{z}} \tilde{J}$, with respect to $\vec{\theta}$, evaluated at $\vec{z}^\star$ and $\overline{\vec{\theta}}$, respectively. Assuming that $\mathcal H$ is positive definite, the quantities
\begin{eqnarray*}
S_i = \left\vert \left\vert \mathcal P \mathcal H^{-1} \mathcal B \vec{e}_i \right\vert \right\vert_{W_{\vec{z}}} \qquad \text{ and } \qquad \tilde{S_i} = \left\vert \left\vert \mathcal P \tilde{\mathcal H}^{-1} \tilde{\mathcal B} \vec{e}_i \right\vert \right\vert_{W_{\vec{z}}} \qquad i=1,2,\dots,n,
\end{eqnarray*}
satisfy
\begin{eqnarray*}
\frac{\vert \tilde{S}_i - S_i \vert }{ \left\vert \left\vert \mathcal H^{-1} \mathcal B \vec{e}_i \right\vert \right\vert_{W_{\vec{z}}}} \le \frac{\vert \vert \mathcal P \vec{n} \vert \vert_{W_{\vec{z}}} }{\vec{s}^T \vec{n} + \alpha},
\end{eqnarray*}
where 
\begin{eqnarray*}
 \qquad \vec{s} = -\frac{\vec{g}}{\vert \vert \vec{g} \vert \vert_2} \qquad \text{and} \qquad \vec{n} = -\mathcal H^{-1} \vec{g}, \qquad \vec{g}=\nabla_{\vec{z}} J(\vec{z}^\star;\overline{\vec{\theta}}).
\end{eqnarray*}
\end{theorem}

Conceptually, Theorem~\ref{thm:compare_S_and_tilde_S} indicates that the relative difference between the $i^{th}$ sensitivity index for perturbed and original optimization problems is given by the size of the projection of the Newton step $\mathcal P \vec{n}$ in the next optimization iterate, divided by the portion of the Newton step pointing in the negative gradient direction plus the length scale parameter $\alpha$. Hence if $\vec{z}^\star$ is optimized in the subspace defined by the range of $\mathcal P$ (the likelihood informed subspace), then $\vert \vert \mathcal P \vec{n} \vert \vert_{W_{\vec{z}}}$ will be small implying that the difference between the theoretically justified sensitivity indices for the perturbed optimization problem and the ill-defined ``sensitivity indices" for the original optimization problem will be small. We may compute $\vert \vert \mathcal P \vec{n} \vert \vert_{W_{\vec{z}}}$ using the eigenvalues and eigenvectors defining the LIS to certify its magnitude in practice. For the permeability inversion problem in Section~\ref{sec:numerics} we have $\vert \vert \mathcal P \vec{n} \vert \vert_{W_{\vec{z}}} = \mathcal O(10^{-4})$. 

Prior to developing the first order update we had no justification for applying HDSA to~\eqref{inv_prob_RS} with sub-optimal solutions. Our formulation of $\tilde{R}$ and Theorem~\ref{thm:compare_S_and_tilde_S} provides the theoretical insights needed to properly define and interpret hyper-differential sensitivity indices for suboptimal solutions. For many practical applications, early termination is common due to the extensive compute time needed to drive the optimizer to convergence. Thanks to this theoretical analysis, HDSA may be used to understand properties of the inverse problem at a modest computational expense.

The length scale parameter $\alpha$ defining $\tilde{R}$ appears throughout our analysis.  Theorem~\ref{thm:vary_alpha} below establishes our previous assertion that the sensitivity indices are robust with respect to changes in $\alpha$. 

\begin{theorem}
\label{thm:vary_alpha}
Let $\tilde{\mathcal H}(\alpha)$ denote the Hessian of $\tilde{J}$ with respect to $\vec{z}$ and $\tilde{\mathcal B}(\alpha)$ denote the Jacobian of $\nabla_{\vec{z}} \tilde{J}$ with respect to $\vec{\theta}$, each evaluated at $\vec{z}^\star$ and $\overline{\vec{\theta}}$, and considered as functions of $\alpha$. Letting $\tilde{S}_i(\alpha)$ be the sensitivity index defined by~\eqref{eqn:sen_indices_formula}, as a function of $\alpha$, and assuming that $\mathcal H$ is positive definite we have 
\begin{eqnarray*}
\frac{\vert \tilde{S}_i(\alpha + \alpha \beta) - \tilde{S}_i(\alpha) \vert }{\left\vert \left\vert \mathcal H^{-1} \mathcal B \vec{e}_i \right\vert \right\vert_{W_{\vec{z}}}} < \vert \beta \vert  \cdot \frac{\vert \vert \mathcal P \vec{n} \vert \vert_{W_{\vec{z}}}}{\vec{s}^T\vec{n}+\alpha(1+\beta)} \qquad \text{ for } -1 < \beta < 1
\end{eqnarray*}
where 
\begin{eqnarray*}
 \qquad \vec{s} = -\frac{\vec{g}}{\vert \vert \vec{g} \vert \vert_2} \qquad \text{and} \qquad \vec{n} = -\mathcal H^{-1} \vec{g}, \qquad \vec{g}=\nabla_{\vec{z}} J(\vec{z}^\star;\overline{\vec{\theta}}).
\end{eqnarray*}
\end{theorem}
Theorem~\ref{thm:vary_alpha} states that for perturbations of $\alpha$ in the form $\alpha (1+\beta)$ for $\beta \in (-1,1)$, changes in the hyper-differential sensitivities are proportional to $\vert \beta \vert$ times $\vert \vert \mathcal P \vec{n} \vert \vert$, which is generally small since $\mathcal P$ projects on to the likelihood informed subspace. This ensures that inferences drawn from HDSA are robust with respect to the length scale parameter $\alpha$. Or in other words, our assumption on length scales defining the norm in~\eqref{R_update_gen_opt} does not have a significant influence on the resulting analysis.

\subsection{Second order a posteriori update}
\label{subsec:second_order_update}

The first order a posteriori update $\tilde{R}$ developed in Subsection~\ref{subsec:first_order_update} addresses the theoretical challenge posed by failing to meet the first order optimality condition $\nabla_{\vec{z}} J(\vec{z}^\star;\overline{\vec{\theta}})=0$. In this subsection, we assume that first order optimality has been satisfied by either the optimizer or the first order update. We address the theoretical challenge posed by failing to meet the second order optimality condition (Assumption~\ref{A2}), that $\mathcal H= \nabla_{\vec{z},\vec{z}} J(\vec{z}^\star;\overline{\vec{\theta}})$ is positive definite. This may occur in nonlinear inverse problems where $\mathcal H$ has eigenvalues which are negative but small in magnitude, for instance, as a result of having limited prior information to define the regularization and terminating the optimization routine prematurely because of ill-conditioning and slow convergence. We propose to overcome this theoretical limitation with another regularization perturbation which we call the second order update. 

Rather than posing an optimization problem to determine the second order update (as was done in~\eqref{R_update_gen_opt} for the first order update), consider the expression~\eqref{eqn:sen_indices_formula} for the LIS-hyper-differential sensitivity indices in Theorem~\ref{thm:sen_indices_likelihood_informed}. We define an update that ensures the Hessian is positive definite while not effecting the terms in~\eqref{eqn:sen_indices_formula}, i.e. the update does not effect the likelihood informed subspace. This will provide a theoretical basis to justify the sensitivity indices but will not require any additional computation.

Theorem~\ref{thm:eig_compare} characterizes the Hessian's positive definiteness in terms of the generalized eigenvalues which define the likelihood informed subspace.
\begin{theorem}
\label{thm:eig_compare}
Let $\mathcal H_M,\mathcal H_R \in \mathbb R^{m \times m}$ be symmetric matrices and $\mathcal H_R$ be positive definite. Then $\mathcal H_M+\mathcal H_R$ is positive definite if and only if the generalized eigenvalues of $(\mathcal H_M,\mathcal H_R)$ are greater than -1, i.e. $\mathcal H_M \vec{v}_j = \lambda_j \mathcal H_R \vec{v}_j$ with $\lambda_j > -1 \hspace{1 mm} \forall j$.
\end{theorem}

If the second order optimality condition is not satisfied, then it follows from Theorem~\ref{thm:eig_compare} that there exist generalized eigenvalue less than or equal to $-1$. Let
$$\lambda_1 \ge \lambda_2 \ge \dots \ge \lambda_K > -1 \ge \lambda_{K+1} \ge \dots \ge \lambda_m,$$
be the generalized eigenvalues of $(\mathcal H_M,\mathcal H_R)$. We ensure positive definiteness of the Hessian by adding an update in the directions of the generalized eigenvectors whose corresponding eigenvalues are less than or equal to $-1$. This update is characterized in Theorem~\ref{thm:gen_reg_update}.
\begin{theorem}
\label{thm:gen_reg_update}
Let $\mathcal H_M,\mathcal H_R \in \mathbb R^{m \times m}$ be symmetric matrices and $\mathcal H_R$ be positive definite. Assume that 
\begin{eqnarray*}
\mathcal H_M \vec{v}_j = \lambda_j \mathcal H_R \vec{v}_j \qquad \text{and} \qquad \vec{v}_j^T\mathcal H_R\vec{v}_j=1, j=1,2,\dots,m.
\end{eqnarray*}
 For some $\delta>0$, define the updates 
 \begin{eqnarray*}
 \mathcal U_i(\delta)=\delta \mathcal H_R \vec{v}_i \vec{v}_i^T \mathcal H_R \qquad i=K+1,K+2,\dots,m.
 \end{eqnarray*}
Then $(\mathcal H_M+\mathcal U_i(\delta))\vec{v}_j=\lambda_j \mathcal H_R \vec{v}_j$, $j\ne i$ and $(\mathcal H_M+\mathcal U_i(\delta))\vec{v}_i = (\lambda_i + \delta) \mathcal H_R\vec{v}_i$.
\end{theorem}

Theorems~\ref{thm:eig_compare} and~\ref{thm:gen_reg_update} suggest that we define the second order update 
\begin{eqnarray*}
\tilde{\tilde{R}}(\vec{z}) = \frac{1}{2} \sum\limits_{i=K+1}^m (\vec{z}-\vec{z}^\star)^T \mathcal U_i\left(-1 - \lambda_i + \epsilon \right) (\vec{z}-\vec{z}^\star) 
\end{eqnarray*}
for an arbitrarily small $\epsilon>0$. Adding $\tilde{\tilde{R}}(\vec{z})$ to $\tilde{J}$ preserves the first order optimality condition since $\nabla_{\vec{z}} \tilde{\tilde{R}}(\vec{z}^\star)=\vec{0}$ and ensures that the second order optimality condition is satisfied at $\vec{z}^\star$.

The generalized eigenvalues less than $-1$ correspond to directions where the misfit has negative curvature and the regularization (which is convex) has positive curvature which is smaller in magnitude. This occurs in directions that are not informed by data and hence are not in the LIS. As Theorem~\ref{thm:gen_reg_update} demonstrates, an update in these directions does not change the leading generalized eigenvectors nor the sensitivity indices, hence it is not necessary to explicitly compute them. This theoretical development serves to justify HDSA when the second order optimality condition is not satisfied, but does not require additional computation. In fact, we do not need to compute the negative generalized eigenvalues, but rather the theory guarantees that ignoring them is not detrimental to our sensitivity analysis. 

\section{Algorithmic overview} \label{sec:alg_overview}
The developments of Sections~\ref{sec:likelihood_informed} and~\ref{sec:enabling_hdsa} are encapsulated in Algorithm~\ref{alg:overview} to provide a concise perspective on the computational components of our proposed analysis which we will refer to as LIS-HDSA. Line~1 solves the PDE-constrained optimization problem. Lines~2 and~3 are the computational steps required for analysis, and Line~4 is a simple postprocessing of the data from Lines~2 and~3 to compute the sensitivities. The subsections below expand on Lines~2 and~3 of Algorithm~\ref{alg:overview}.

The computational cost is dominated by the PDE solves required to compute matrix-vector products with $\mathcal H_M$ in Line~2 and with $\mathcal B$ in Line~3. The cost of the first order update is negligible and the second order update is a strictly theoretical tool and does not require computation. 
Note that Line~2 uses the first order update regularization $\mathcal H_R+\mathcal H_{\tilde{R}}$ rather than $\mathcal H_R$ which was presented in~\eqref{eqn:likelihood_informed_gen_eig}. This is because the LIS projection was presented before the a posteriori updates for clarity of the exposition, but it is theoretically appropriate to analyze the perturbed optimization problem. However, since adding the first order update increases the regularization in the gradient direction, which we assume is uninformed, the LIS for the perturbed problem will be close to (or even equal to) the LIS for the original problem. 
\begin{algorithm}
\caption{Computation of LIS-hyper-differential sensitivity indices}
\label{alg:overview}
\begin{algorithmic}[1]
\STATE Solve the optimization problem (possibly with early termination) $$\min\limits_{\vec{z} \in \mathbb R^m} J(\vec{z},\overline{\vec{\theta}}):=M(\vec{z},\overline{\vec{\theta}}) + R(\vec{z})$$ \vspace{-3 mm}
\STATE Compute the leading positive eigenpairs $\mathcal H_M \vec{v}_j = \lambda_j (\mathcal H_R + \mathcal H_{\tilde{R}}) \vec{v}_j,$ $j=1,\dots,r$
\STATE Compute $\mathcal B e_i = \nabla_{\vec{z},\vec{\theta}} J(\vec{z}^\star; \overline{\vec{\theta}}) e_i$, $i=1,2,\dots,n$
\STATE Compute the hyper-differential sensitivity indices using~\eqref{eqn:sen_indices_formula}
\end{algorithmic}
\end{algorithm}

\subsection{Hessian generalized eigenvalue problem} \label{ssec:hgev}
In general, the misfit Hessian is only accessible via matrix-vector products so the Hessian generalized eigenvalue problem (GEVP), Line~2 of Algorithm~\ref{alg:overview}, may be solved with either iterative methods (such as Krylov solvers) or randomized methods. Each matrix-vector product with the misfit Hessian requires two PDE solves (assuming that adjoints are used \cite{sunseri_hdsa}). Randomized methods afford greater parallelism than iterative methods when adequate computational resources are available \cite{HDSA,saibaba_gsvd}. Algorithm~\ref{alg:GHEP} adapts Algorithm~6 in \cite{arvind} by iterating on the number of desired eigenvalues. It terminates on a user specified minimum eigenvalue, which may be chosen using the eigenvalues interpretation as the ratio of misfit and regularization~\eqref{eqn:eig_interpretation}. 

\begin{algorithm}
\caption{Randomized Generalized Hermitian Eigenvalue Algorithm}
\label{alg:GHEP}
\begin{algorithmic}[1]
\STATE \textbf{Input: } oversampling factor $p \in \mathbb N$, initial target rank $r_0 \in \mathbb N$, rank increment $\Delta r \in \mathbb N$, minimum eigenvalue threshold $\lambda_{min} \in \mathbb R$
\STATE Set $\lambda_{iter} = \infty$
\STATE Set $r=r_0$
\STATE Generate a random matrix $\Omega \in \mathbb R^{n \times (r_0+p)}$
\WHILE{$\lambda_{iter} > \lambda_{min}$}
\IF{Number of columns of $\Omega < r+p$}
\STATE $\Omega = [\Omega, \Omega_{\Delta r}]$ for a randomly generated $\Omega_{\Delta r} \in \mathbb R^{n \times \Delta r}$
\STATE Set $r = r + \Delta r$
\ENDIF
\STATE Compute $Y =( \mathcal H_R + \mathcal H_{\tilde{R}})^{-1} \mathcal H_M \Omega$
\STATE Compute $Q$, whose columns span the range of $Y$, with $Q^T ( \mathcal H_R + \mathcal H_{\tilde{R}}) Q=I$
\STATE Compute $T=Q^T \mathcal H_M Q$
\STATE Compute the eigen decomposition $T=S \Lambda S^T$ (with decreasing eigenvalues on the diagonal of $\Lambda$)
\STATE Set $\lambda_{iter}=(\Lambda)_{r,r}$ (the $r^{th}$ eigenvalue in $\Lambda$)
\ENDWHILE
\STATE Compute $\vec{v}_j = Q \vec{s}_j$, $j=1,2,\dots,r$, where $\vec{s}_i$ is the $j^{th}$ column of $S$ (normalized so that $\vec{v}_j^T ( \mathcal H_R + \mathcal H_{\tilde{R}}) \vec{v}_j=1$)
\STATE \textbf{Return: } Estimated generalized eigenvalues and eigenvectors $\{\lambda_j,\vec{v}_j \}_{j=1}^r$
\end{algorithmic}
\end{algorithm}

The randomized GEVP algorithm sketches the range space of $( \mathcal H_R + \mathcal H_{\tilde{R}})^{-1} \mathcal H_M \in \mathbb R^{m \times m}$ by applying it to a collection of independent random vectors. This defines a low dimensional subspace from which the generalized eigenvalues and eigenvectors may be estimated through direct methods. The accuracy of randomized solvers for a given number of matrix-vector products is slightly poorer (though in many cases not by much) than a Krylov method with a comparable number of matrix-vector products; however, the randomized algorithm allows asynchronous computation which may reduce wall clock time.

The input $p$ (oversampling factor) ensures that the subspace defined by the leading $r$ eigenvectors is well approximated by the range of $( \mathcal H_R + \mathcal H_{\tilde{R}})^{-1} \mathcal H_M \Omega$. The oversampling factor is well understood and is typically taken to be on the order of 10-20 \cite{randomized_la_review,arvind} to give an accurate approximation of the leading eigenvectors with high probability. Since the generalized eigenvalues appear in the denominator in \eqref{eqn:sen_indices_formula}, and the smallest generalized eigenvalues/vectors are the hardest to estimate, we suggest taking $p=20$ to ensure reliable computation. Assuming $L$ processors are available to execute parallel PDE solves, we set $r_0=L-p$ and $\Delta r= L$ to execute all matrix-vector products simultaneously. 

The while loop terminates on a minimum eigenvalue threshold $\lambda_{min}$ which the user chooses a priori using the interpretation of the eigenvalues \eqref{eqn:eig_interpretation}. Setting $\lambda_{min}=1$ projects on the subspace where the data contributes more than the regularization. Since there is some ambiguity in the choice of $\lambda_{min}$ we propose an approach to assess the robustness of the sensitivity indices with respect to changes in $\lambda_{min}$. Specifically, the user may execute Algorithm~\ref{alg:GHEP} with $\lambda_{min}$ less than their desired threshold. The LIS-hyper-differential sensitivities may be computed via \eqref{eqn:sen_indices_formula} for different choices of $r$ (corresponding to different eigenvalue thresholds) at a negligible computational cost. We demonstrate this in Section~\ref{sec:numerics}. 

The computational cost of the Algorithm~\ref{alg:GHEP} is dominated by the products involving $\mathcal H_M$ in Lines~10 and~12. These lines may reuse matrix-vector products from previous iterations of the while loop so the total number of matrix-vector products involving $\mathcal H_M$ is $2(r+p)$, where $r$ is the target rank at the termination of the while loop. Lines~10 and~12 are embarrassingly parallel due to the randomization decoupling the vectors, hence the cost is mitigated when parallelization is enabled.

The random matrix generation in Lines~4 and~7 may be done in a variety of ways but most commonly have entries sampled independently from a standard normal distribution. The orthogonalization in Line~11 may be done with a combination of Cholesky and QR matrix decompositions, see Algorithms~4 and~5 in \cite{arvind} for details. The eigen decomposition in Line~13 is on a small dense matrix and may be done using classical dense linear algebra kernels.

\subsection{Action of $\mathcal B$}
If $\{\mathcal B e_i\}_{i=1}^n$ are computed by taking $n$ matrix vector products, the computational cost for Line~3 of Algorithm~\ref{alg:overview} is $2n$ PDE solves \cite{sunseri_hdsa} ($n$ is the dimension of the auxiliary parameters $\theta$). For moderate $n$, this cost is easily mitigated by parallelism of the matrix-vector products (it is embarrassingly parallel for up to $n$ processors). For large values of $n$, which is common when auxiliary parameters are spatially or temporally distributed, Line~3 becomes computationally intensive. However, when $n$ is large there is frequently low rank structure in the auxiliary parameter space which may be exploited by computing its Singular Value Decomposition. This has the potential to significantly reduce the number of matrix-vector products. 

\section{Numerical results} \label{sec:numerics}
This section demonstrates the use of LIS-HDSA for a prototypical permeability inversion using a tracer test. Such problems are critical in ground water and petroleum reservoir management which are characterized by spatial heterogeneity, data sparsity, and complex nonlinearities. Our prototypical inverse problem emulates such features that are common in more complex material property estimation and source identification problems.

\subsection{Inverse problem formulation} \label{ssec:model}
We consider inverting for a log permeability field $\kappa$ via a tracer test. Specifically, we use observations of the tracer concentration $c$ and fluid pressure $p$ which are modeled (in two dimensions) by the advection diffusion equation and Darcy's equation, respectively. We consider the optimization problem
 \begin{align}
\label{eqn:perm_inv}
 \min_{\kappa} \sum\limits_{i=1}^{N} \frac{2}{w_c^2} \vert \vert \mathcal T_c c(\kappa,t_i) - \vec{d}_c(t_i )\vert \vert_2^2 + \frac{1}{w_p^2} \vert \vert \mathcal T_p p(\kappa) - \vec{d}_p\vert \vert_2^2+ \gamma_1 \vert \vert \nabla \kappa \vert \vert^2 +  \gamma_2 \vert \vert \kappa \vert \vert^2 
 \end{align}
where $\mathcal T_c$ and $\mathcal T_p$ denote observation operators, $\vec{d}_c(t_i)$ and $\vec{d}_p$ denotes the observed data for concentration (at time nodes $t_i$, $i=1,2,\dots,N$) and pressure, and $(c(\kappa,t),p(\kappa))$ satisfy the system of PDEs 
\begin{align*}
& -\nabla\cdot(e^\kappa \nabla p) = 0 \quad &
\text{ in } D \\
& \frac{\partial c}{\partial t} - \nabla\cdot\Big( \eta(\vec{\theta}) \nabla c \Big) + \nabla\cdot\big( -e^\kappa \nabla p c\big)
= g(\vec{\theta}) \quad & \text{ in } [0,T]\times D \\
& p = p_1(\vec{\theta}) \quad & \text{on } \Gamma_1 \\
& p = p_2(\vec{\theta}) \quad & \text{on } \Gamma_3 \\
& e^\kappa \nabla p \cdot n = 0 \quad & \text{ on } \Gamma_0 \cup \Gamma_2 \\
& \nabla c \cdot n = 0 \quad  & \text{ on } [0,T]\times\{\Gamma_0\cup\Gamma_1\cup\Gamma_2\cup\Gamma_3\} \\
& c(0,x) = 0 \quad & \text{ in } D
\end{align*}
on the domain $D=(0,1)^2$ with $\Gamma_i$, $i=0,1,2,3$, denoting the four sides of the square (depicted in Figure~\ref{fig:domain_sensors}).

The PDE system has uncertainty in its pressure Dirichlet conditions $p_1$ and $p_2$, its tracer source $g$, and its diffusion coefficient $\eta$. Each are parameterized by $\vec{\theta}$ so that $\overline{\vec{\theta}}=\vec{0}$ corresponds to a nominal estimate of the parameter. 

The pressure Dirichlet conditions are spatially varying and their uncertainty is represented using linear finite element basis functions defined on each boundary. The tracer source consists of 16 injection locations with its uncertainty parameterized by taking linear finite element basis functions defined on 3x3 square meshes around each injection point. The diffusion coefficient is a scalar and hence its uncertainty is parameterized by a scalar. Parameterizations are given in Table~\ref{tab:pde_params}.

	\begin{table}[!ht]
	    \centering
	    \begin{tabular}{c|c}
	    \hline
	    Dirichlet Condition & $p_1(\vec{\theta}) = \left(15 + \cos(2 \pi y ) + .5\cos(4 \pi y)\right) \delta_{p_1}(\vec{\theta})$ \\
	    Dirichlet Condition & $p_2(\vec{\theta}) = \left(10 + 2\cos(2 \pi y)\right) \delta_{p_2}(\vec{\theta})$ \\
	    Dirichlet Perturbation & $\delta_{p_1}(\vec{\theta}) = 1 + 0.1 \sum_{j=1}^{21} \theta_{j} \phi_j$ \\
	    Dirichlet Perturbation & $\delta_{p_2}(\vec{\theta}) = 1 + 0.1 \sum_{j=1}^{21} \theta_{21+j} \phi_j$ \\
	    Source Term & $g(\vec{\theta}) = \sum_{k=1}^{16} 10e^{-100((x-v_k)^2 + (y-w_k)^2)} \delta_{s_k}(\vec{\theta})$ \\
	    Source Perturbations & $\delta_{s_k}(\vec{\theta}) = 1 + 0.1 \sum_{j=1}^9 \theta_{(42+(k-1)9+j)} \psi_j(v_k,w_k)$ \\
	    Diffusion Coefficient & $\eta(\vec{\theta}) = .025 (1+0.1 \theta_{187})$ \\
	     \hline
	     \end{tabular}
	    \caption{Expressions for the PDE parameters with uncertainty modeled by $\vec{\theta}$. The functions $\{\phi_j\}_{j=1}^{21}$ are linear finite element basis functions defined on a uniform mesh of $[0,1]$. The functions $\{\psi_j(v_k,w_k)\}_{j=1}^9$ are linear finite element basis functions defined on a $3x3$ square mesh centered at $(v_k,w_k)$. The set of points $\{(v_k,w_k)\}_{k=1}^{16}$ correspond to the injection locations depicted by the diamonds in Figure~\ref{fig:domain_sensors}.}
	    \label{tab:pde_params}
	\end{table}

To avoid the inverse crime, we generate data by solving the forward problem on a fine mesh, add zero-mean Gaussian white noise to the data, and sparsify in space (the observation locations are depicted in Figure~\ref{fig:domain_sensors}). The objective function weights $w_c$ and $w_p$ nondimensionalize the misfit so that pressure and concentration data is comparable, and the regularization coefficients $\gamma_1$ and $\gamma_2$ are chosen to avoid over smoothing the solution. The concentration data is multiplied by $2$ to put greater weight on fitting it. This improves the inversion by emphasizing the advective information in the data. Problem parameter details are in Table~\ref{tab:prob_params}.

	\begin{table}[!ht]
	    \centering
	    \begin{tabular}{c|c}
	    \hline
	    Spatial mesh for data generation & $128 \times 128$ rectangular mesh\\
	    Spatial mesh for inverse problem & $64 \times 64$ rectangular mesh \\
            Time steps for data generation & $151$ \\
            Time steps for inverse problem & $76$ \\
            Noise standard deviation & $1\%$ of the node data value \\
            Final time & $T=0.25$ \\
            Discretized dimensions & $\vec{u}  \in \mathbb R^{325325}$, $\vec{z} \in \mathbb R^{4225}$, $\vec{\theta} \in \mathbb R^{187}$ \\
            Observation locations & Depicted in Figure~\ref{fig:domain_sensors} \\
             Regularization coefficients & $\gamma_1=10^{-5}$ and $\gamma_2=10^{-7}$ \\
            Misfit weights $w_c$ and $w_p$ & Average concentration and pressure data \\
	     \hline
	     \end{tabular}
	    \caption{Discretization, data generation, and objective function parameters defining the inverse problem.}
	    \label{tab:prob_params}
	\end{table}

\begin{figure}[h]
\centering
  \includegraphics[width=0.65\textwidth]{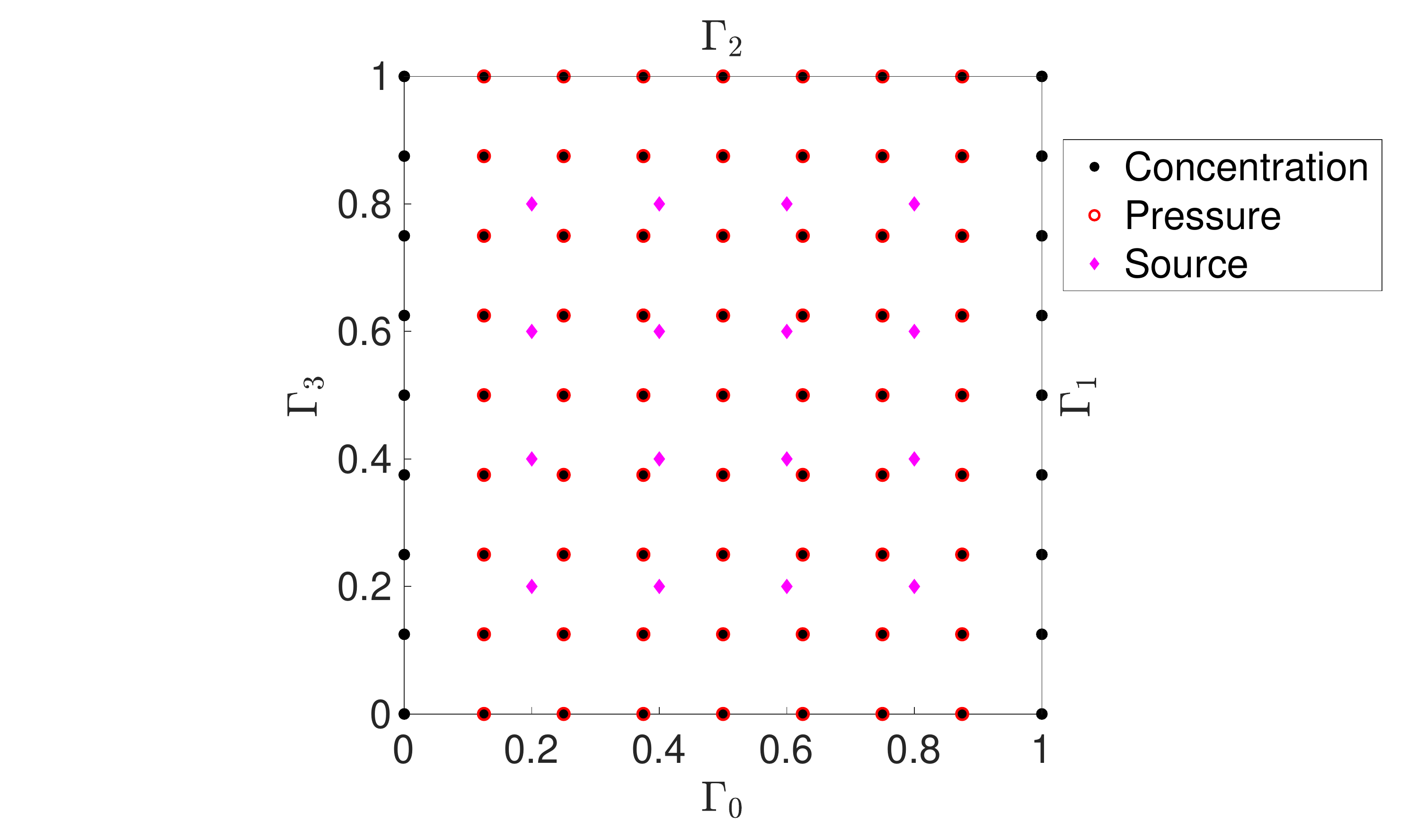}
  \caption{Depiction of the computational domain with locations of concentration sensors (dots), pressure sensors (circles), and source injection sites (diamonds).}
  \label{fig:domain_sensors}
\end{figure}

\subsection{Permeability field estimate} \label{ssec:solution}
The optimization problem is solved with $\vec{\theta}=\vec{0}$ using a truncated conjugate gradient trust region algorithm. Table~\ref{tab:iteration_history} displays the iteration history. The slow convergence observed in this example is typical for large-scale nonlinear inverse problems. Comparing the solutions after 75 iterations and 125 iterations shows that the estimate is not noticeably different, despite the fact that it took significantly more computation time to run these 50 iterations. Furthermore, based on the trends in Table~\ref{tab:iteration_history}, it may take many additional iterations to convergence (significant wall clock time) without improving the quality of the solution. This challenge exemplifies our motivation to apply HDSA to suboptimal solutions rather than requiring convergence to justify the analysis.

	\begin{table}[!ht]
	    \centering
	    \begin{tabular}{c|c|c|c}
	    Iteration & Objective & Gradient Norm & Step Size \\ \hline
	    0 & 17.2 & 1.33 &  N/A  \\
	    4 & 9.59 & .697 &  15.6  \\
	    10 & 3.29 & .676 &  2.38  \\
	     41 & .897 & .113 &  2.02  \\
	     65 & .578 & .331 &  .115  \\
	     75 & .571 & .102 &  .109  \\
	     125 & .529 & .081 &  .034  \\
	     \end{tabular}
	    \caption{Iteration history for the optimization problem.}
	    \label{tab:iteration_history}
	\end{table}

The initial iterate, estimated log permeability field, and the ``true" log permeability field which generated the data are given in Figure~\ref{fig:perm_fields}. We observe that the high and low permeability regions in the middle of the domain are apparent in the solution, while many of the fine scale features are not resolved. This is unsurprising given the diffusivity and nonlinearity of the physics. 

\begin{figure}[h]
\centering
  \includegraphics[width=0.3\textwidth]{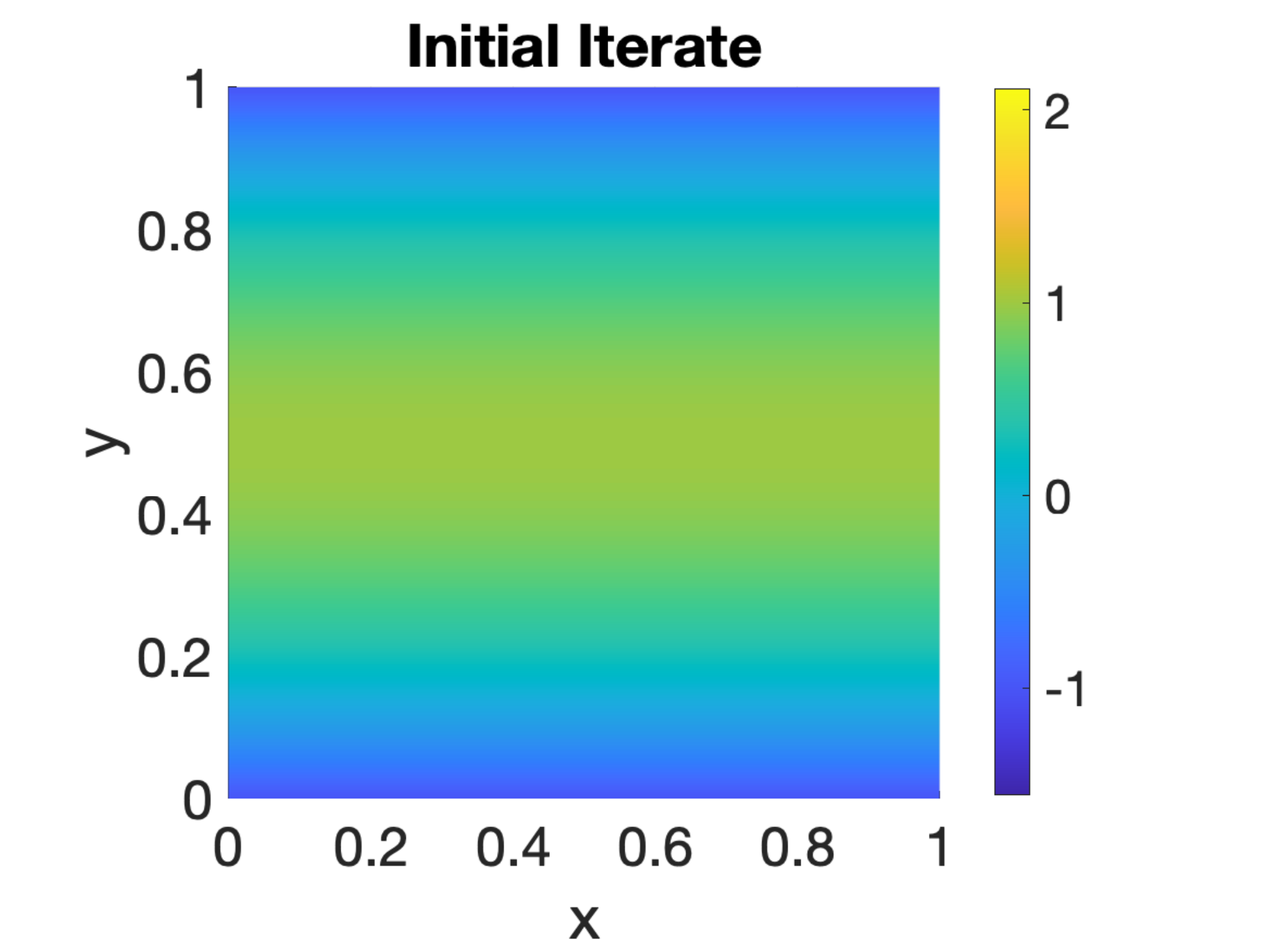}
    \includegraphics[width=0.3\textwidth]{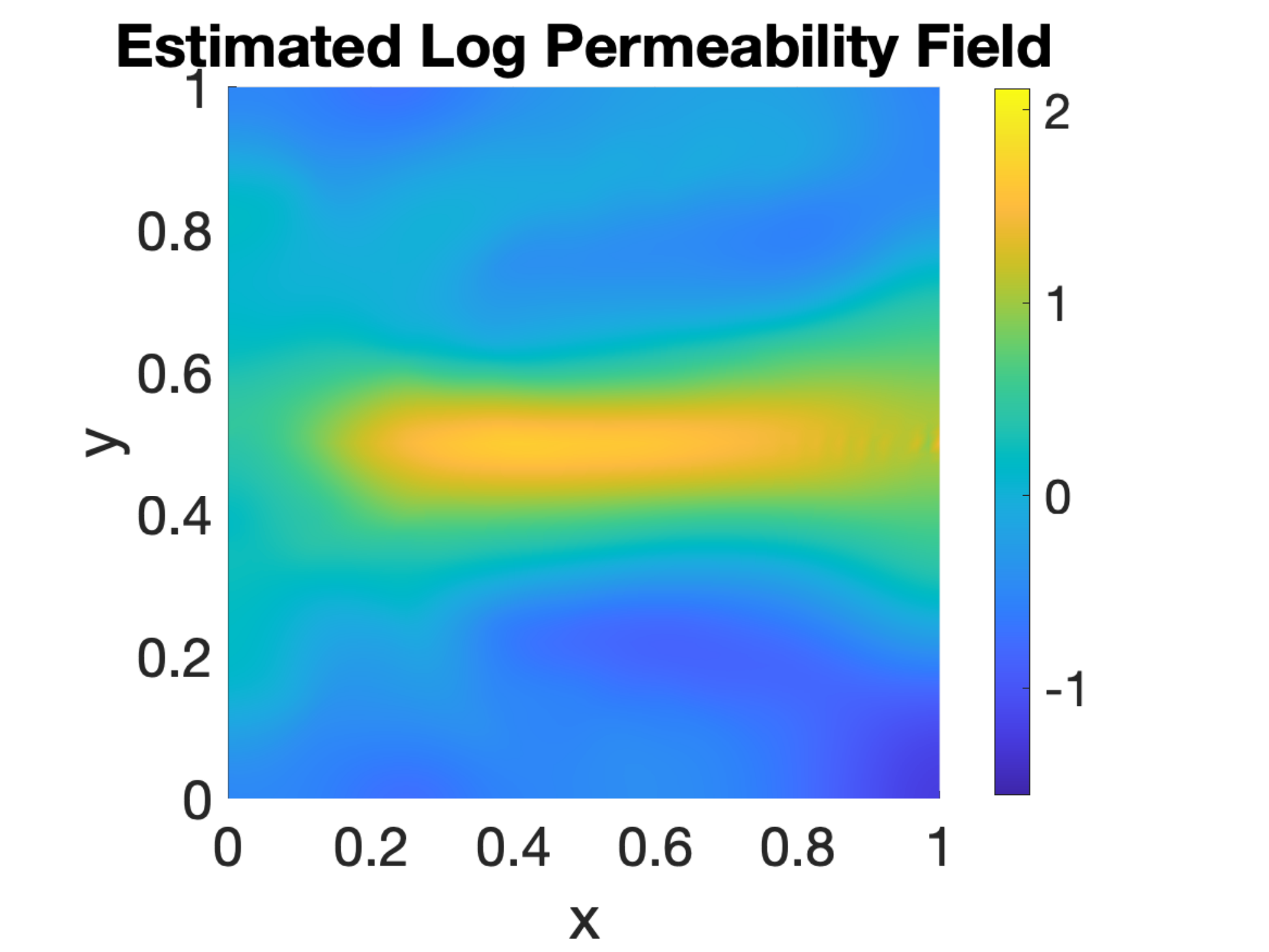}
      \includegraphics[width=0.3\textwidth]{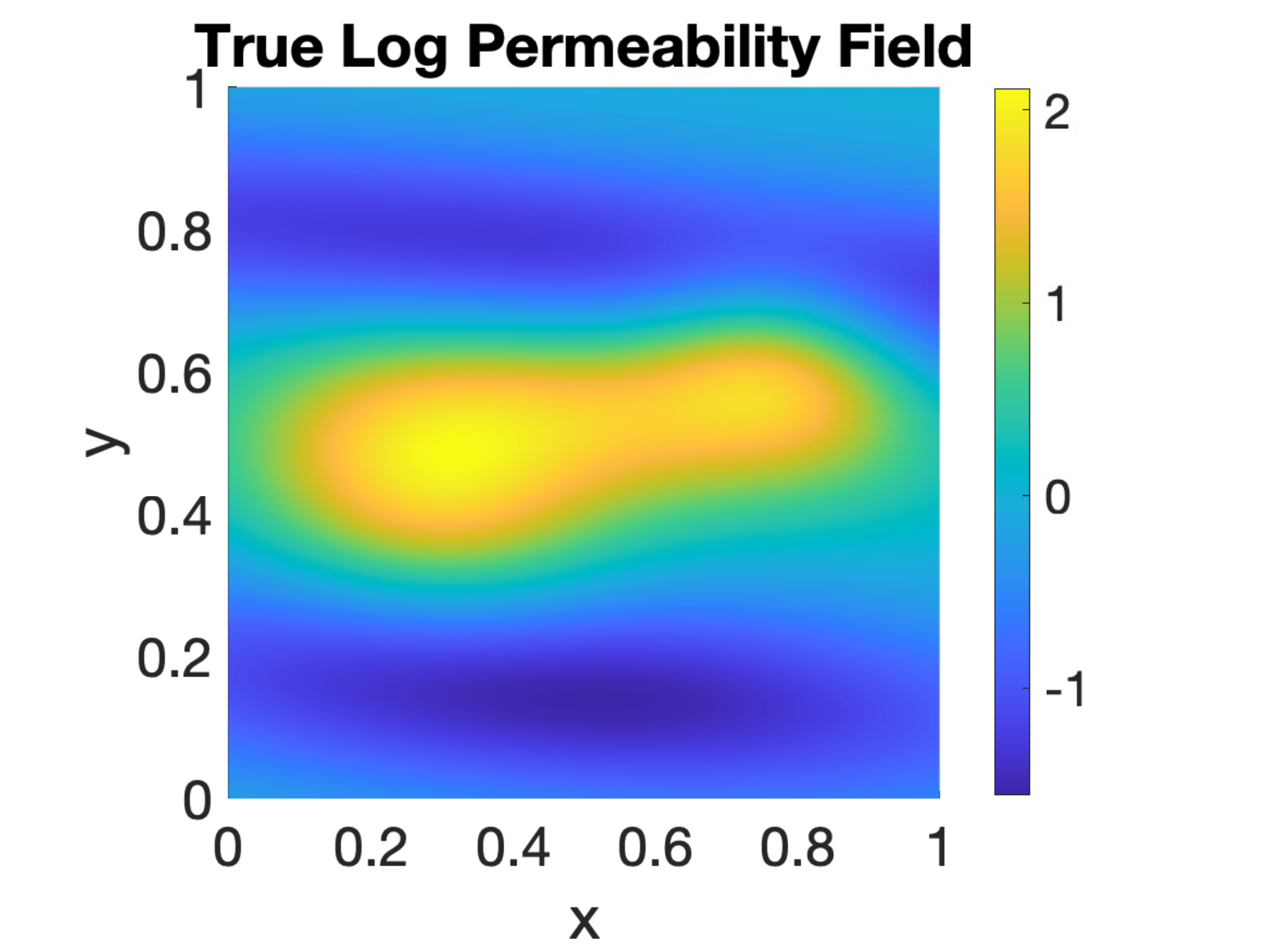}
  \caption{Left: initial iterate for the inverse problem; center: estimated log permeability field (after 75 iterations); right: true log permeability field used to generate the data.}
  \label{fig:perm_fields}
\end{figure}

Figure~\ref{fig:states} displays (top row) the pressure and time snapshots of the concentration computed using the estimated log permeability field alongside (bottom row) the corresponding data, $\{\vec{d}_c(t_i),\vec{d}_p\}$, plotted as a field rather than sparse data. The pressure Dirichlet conditions induce a flow from right to left. As a result of the high permeability region around $y=0.5$, the tracer, which enters the domains through sixteen injection sites, is advected toward two outflow regions around $(0,0.4)$ and $(0,0.6)$. The regions of greatest estimation error in Figure~\ref{fig:perm_fields} are unsurprising given the characteristics of the flow in Figure~\ref{fig:states}. Comparing the good fit of the state data in Figure~\ref{fig:states} with the log permeability field error in Figure~\ref{fig:perm_fields} highlights the ill-posedness of the problem, a characteristic commonly observed in practice.

\begin{figure}[h]
\centering
  \includegraphics[width=0.24\textwidth]{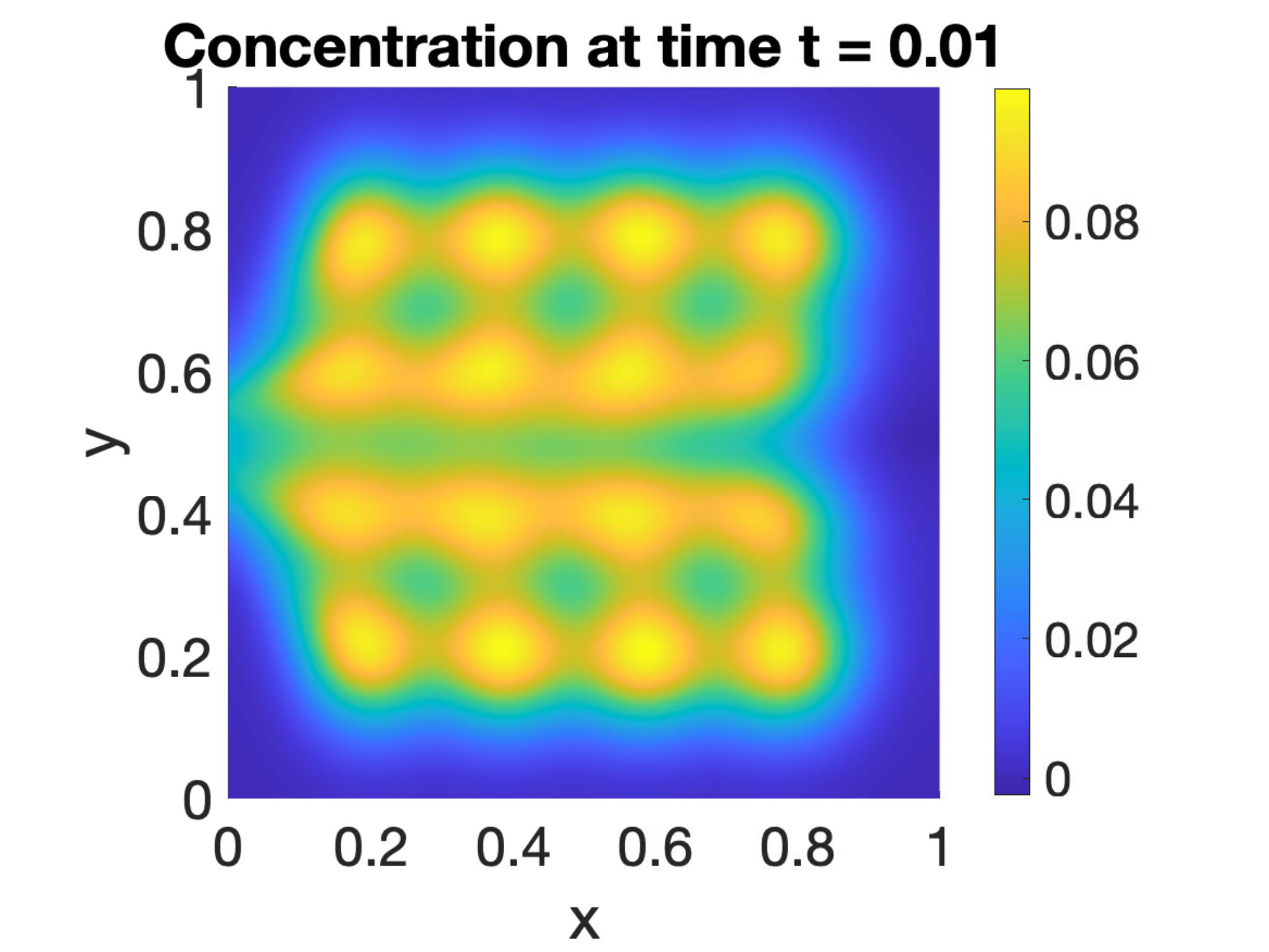}
    \includegraphics[width=0.24\textwidth]{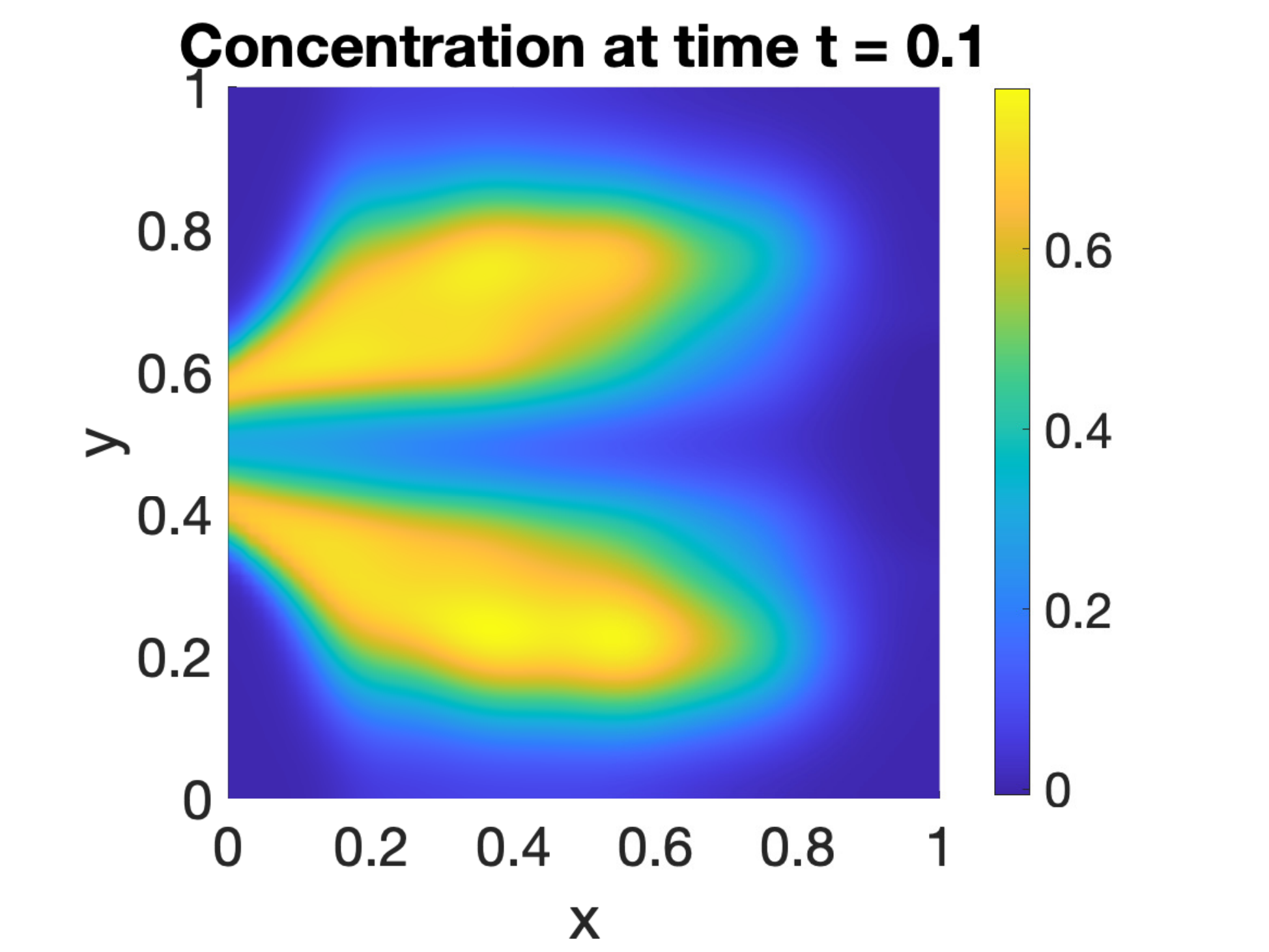}
      \includegraphics[width=0.24\textwidth]{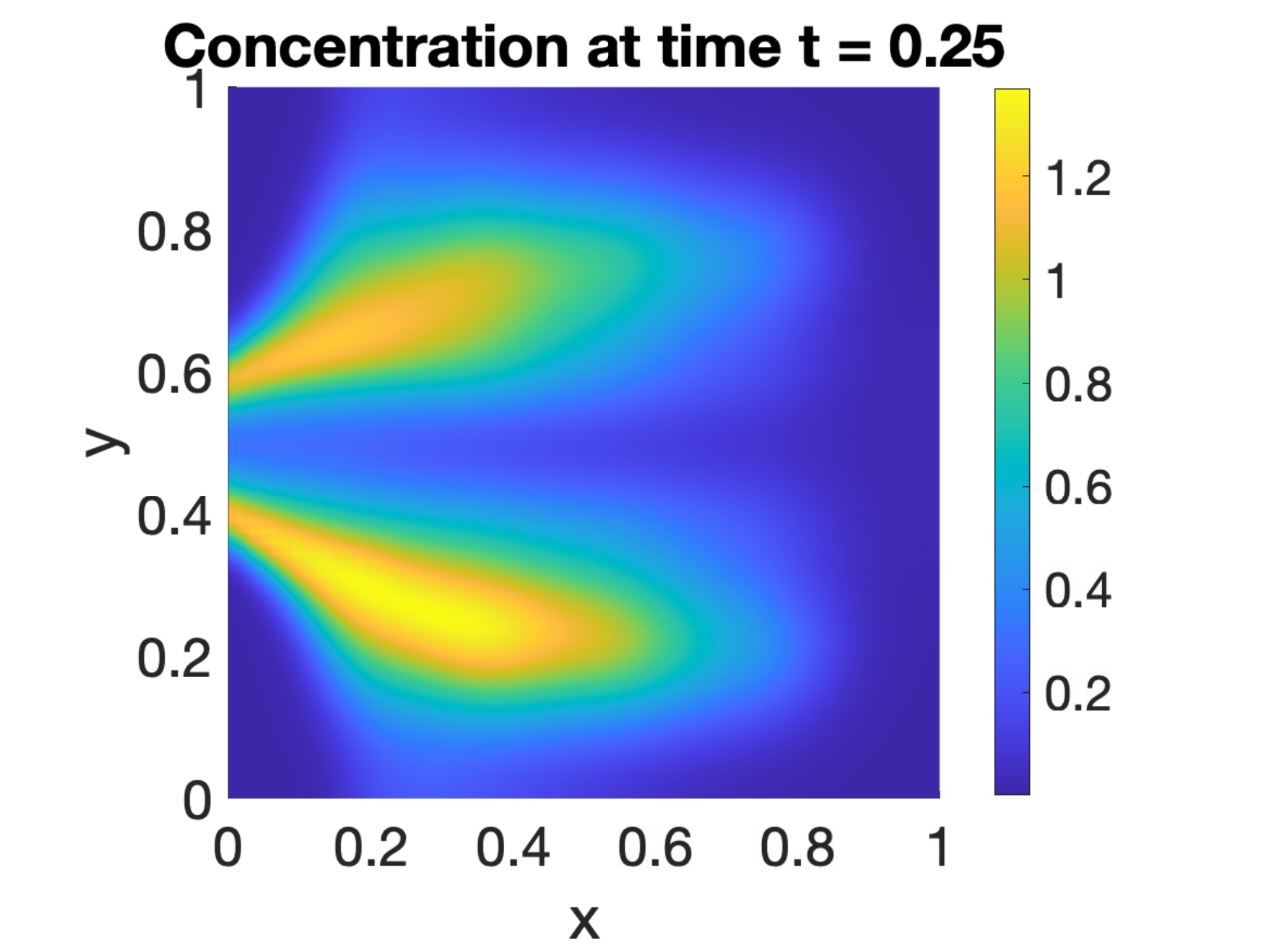}
       \includegraphics[width=0.24\textwidth]{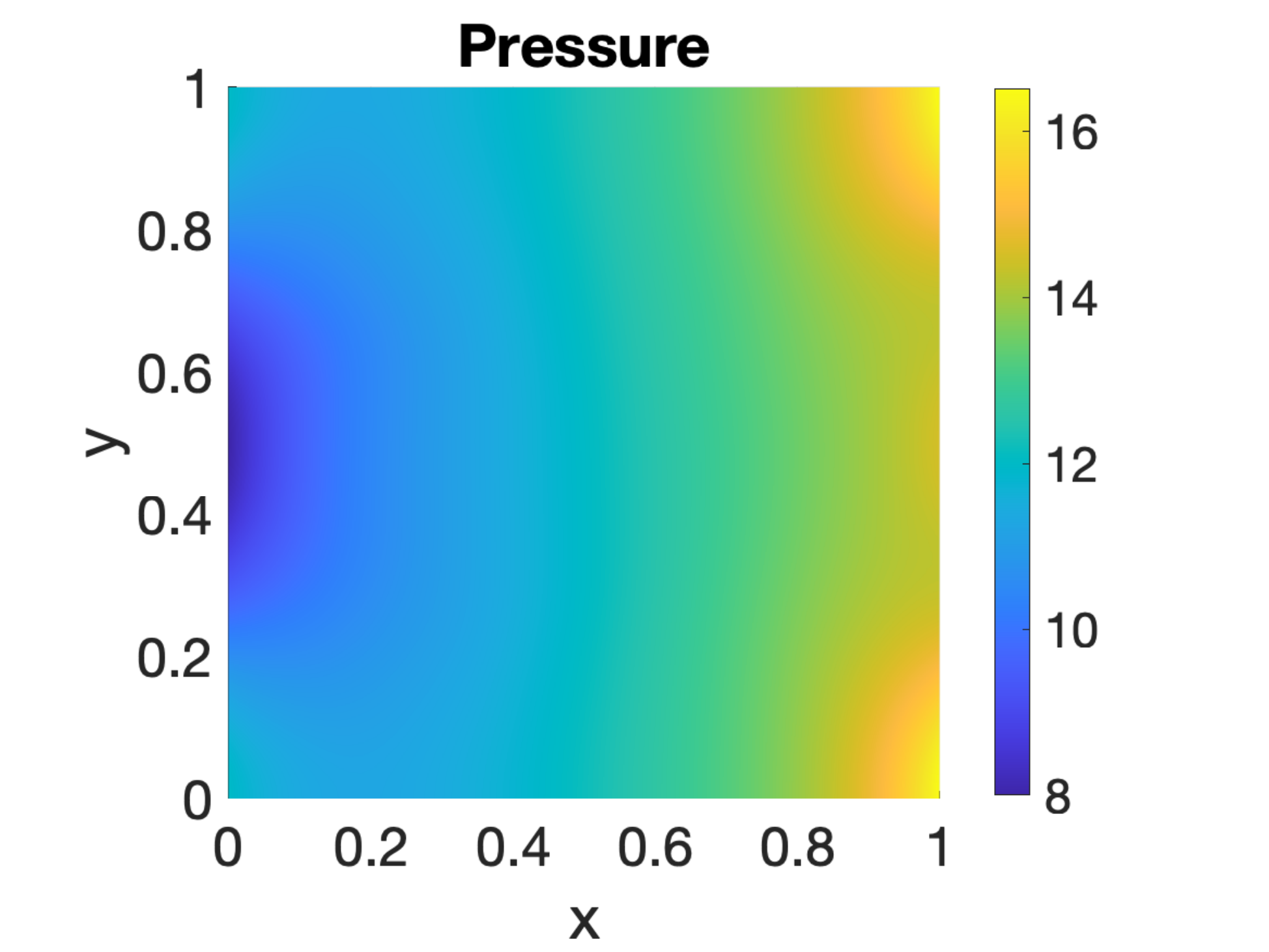}\\
         \includegraphics[width=0.24\textwidth]{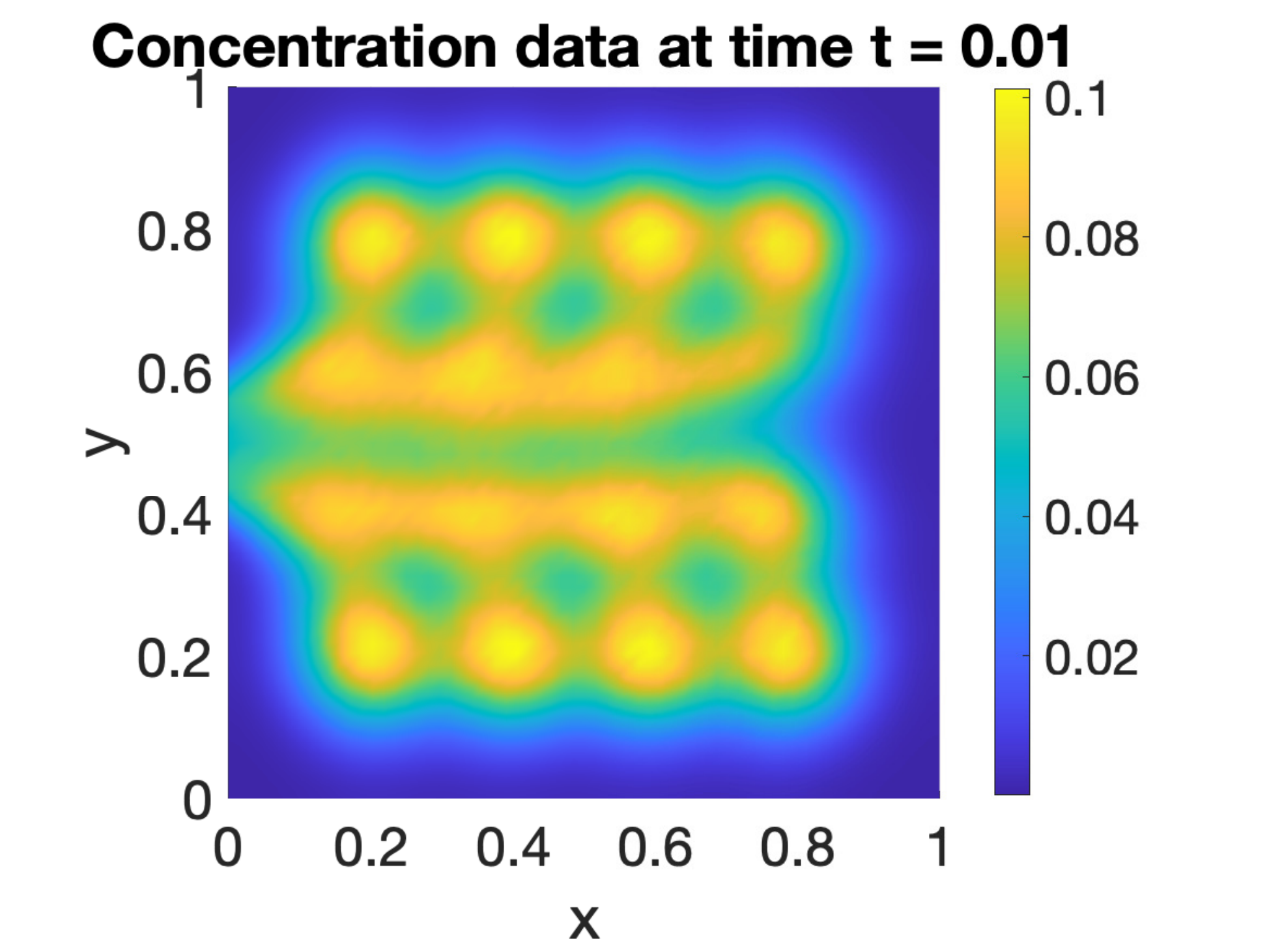}
    \includegraphics[width=0.24\textwidth]{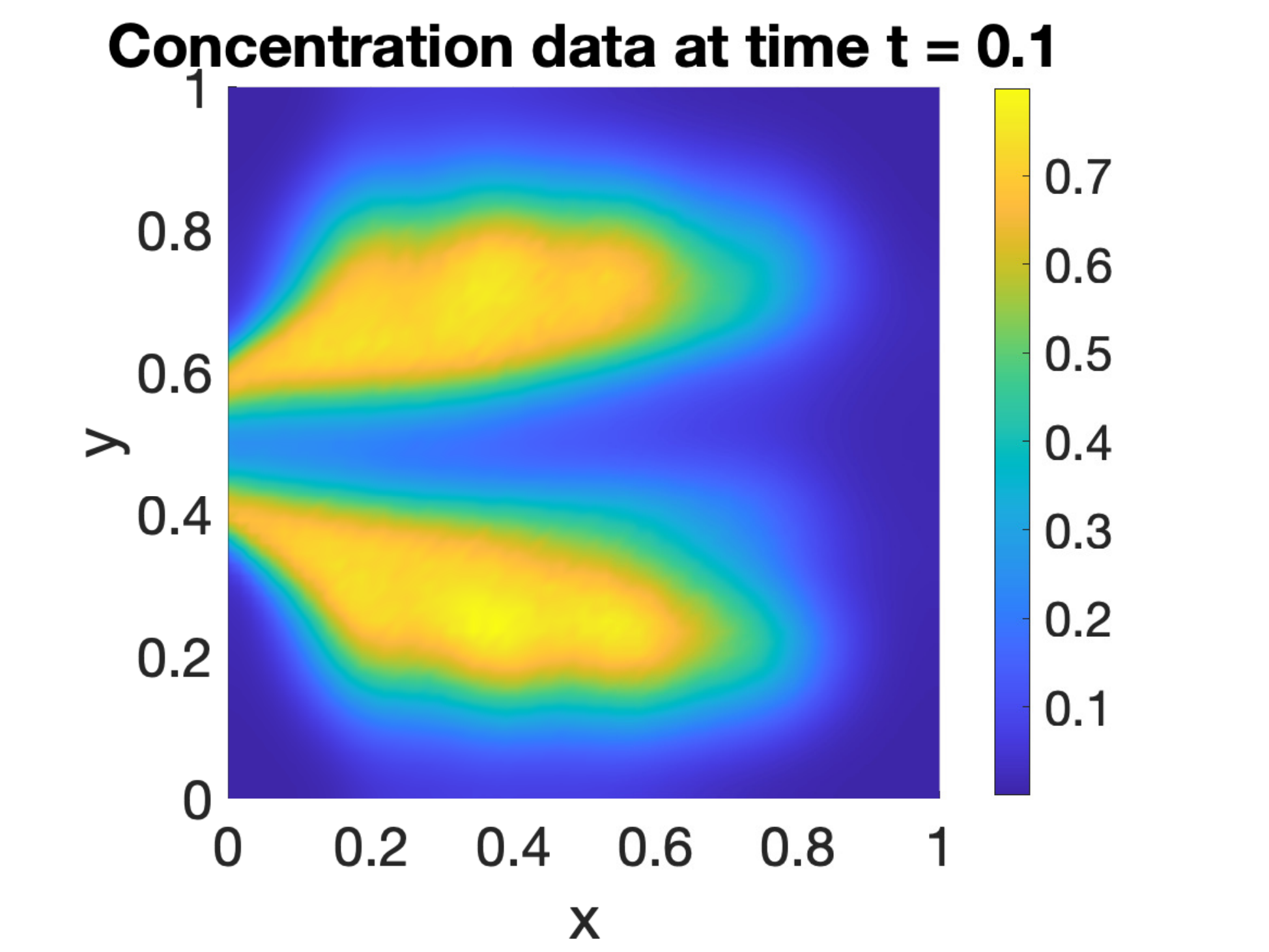}
      \includegraphics[width=0.24\textwidth]{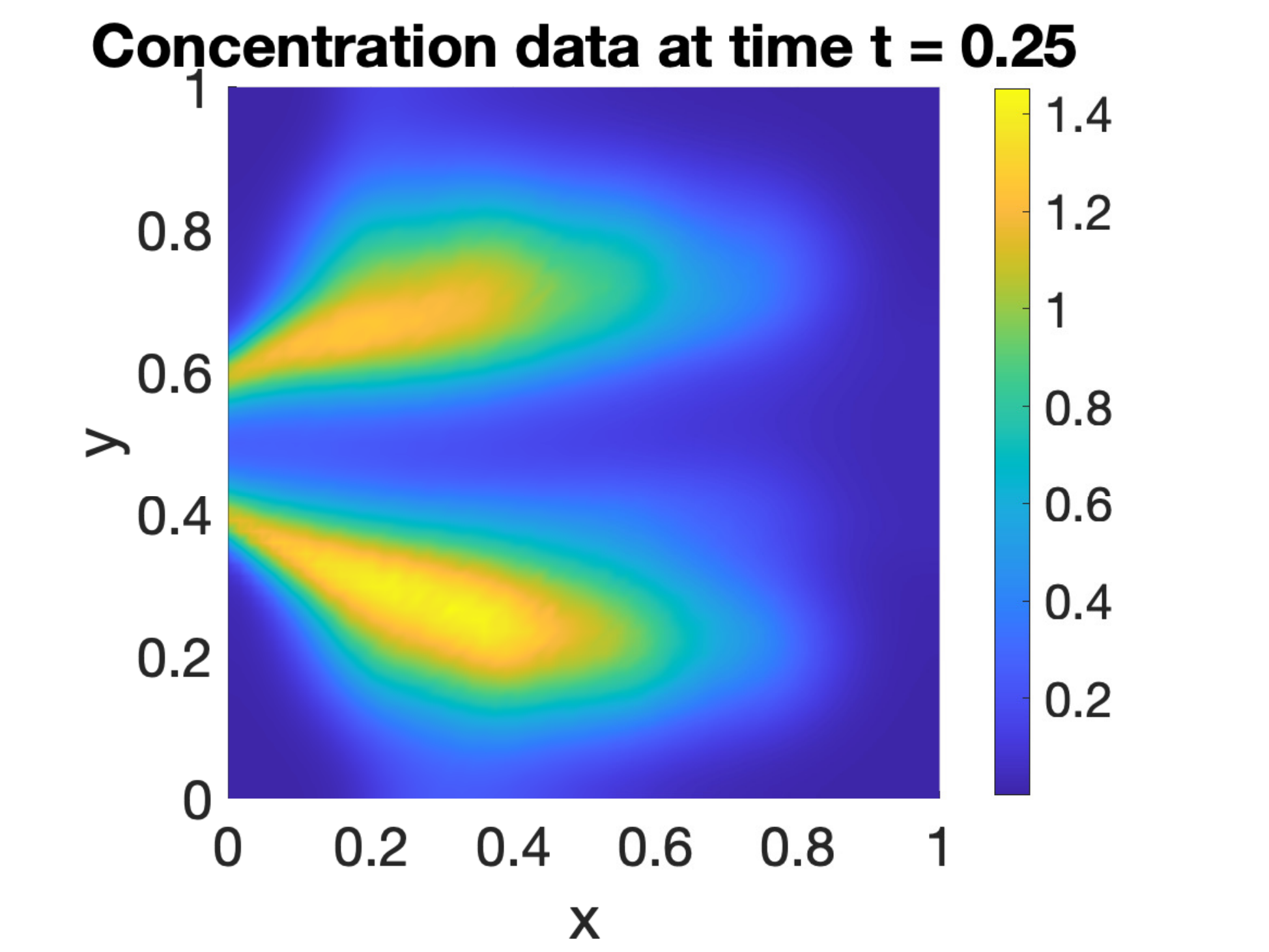}
       \includegraphics[width=0.24\textwidth]{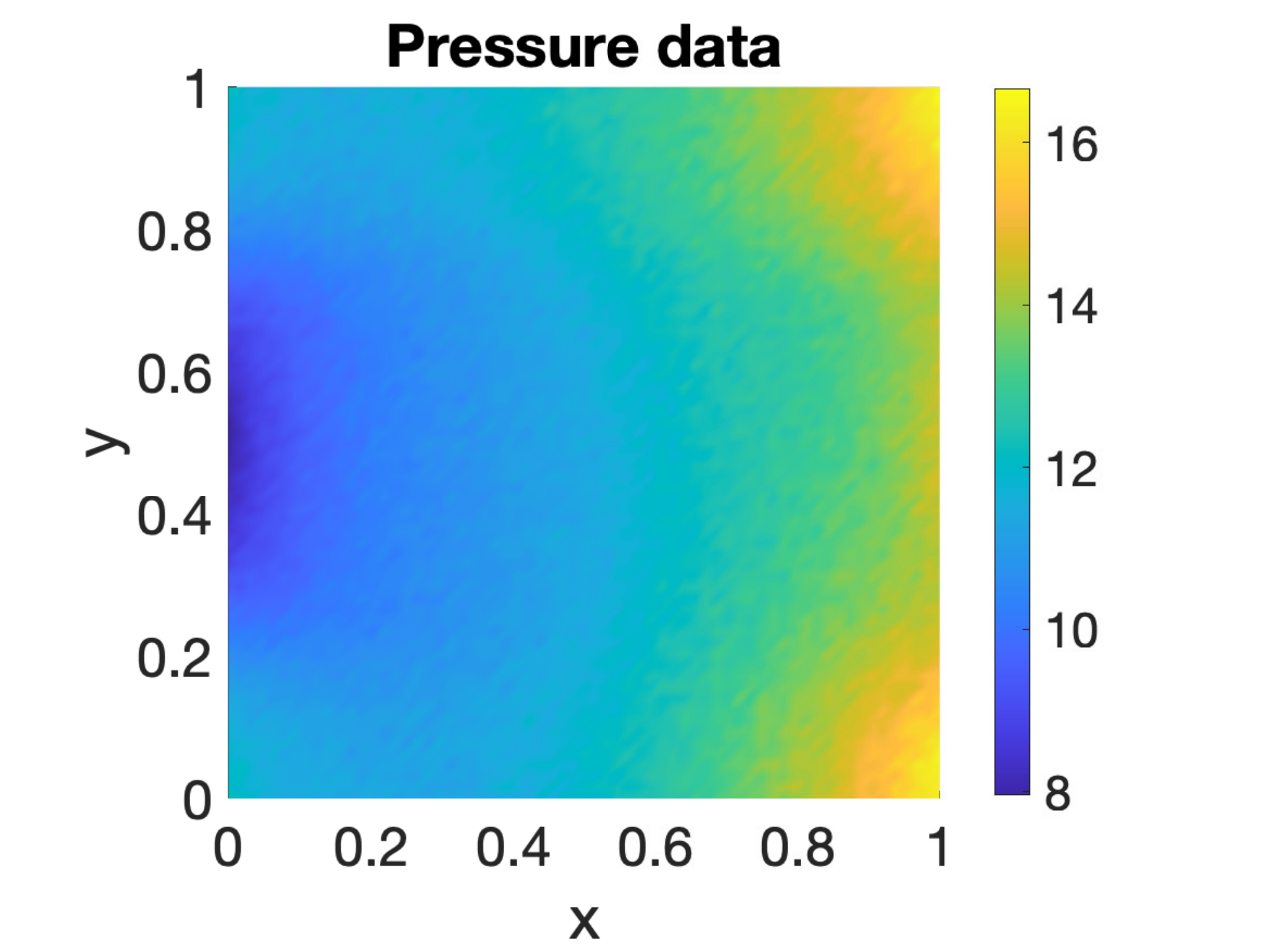}
  \caption{Pressure and time snapshots of concentration. The top row is computed by using the estimated log permeability field (center panel of Figure~\ref{fig:perm_fields}) while the bottom row is computed by using the true log permeability field (right panel of Figure~\ref{fig:perm_fields}) and adding noise. From left to right displays the time evolution of the concentration with the pressure in the rightmost panel.}
    \label{fig:states}
\end{figure}

\subsection{HDSA results} \label{ssec:hdsa}
There is no theoretical justification to use HDSA on the solution of this inverse problem without the developments of Section~\ref{sec:enabling_hdsa}. The suboptimal solution only reduces the gradient norm by two orders of magnitude despite significant computational efforts, and the Hessian is indefinite due to insufficient information in the regularization. This serves as an illustrative and motivating application for the contributions of this article.

The length scale parameter is set to $\alpha=0.5(\text{max}(\vec{z}^\star)-\text{min}(\vec{z}^\star))=1.4657$ as discussed in Subsection~\ref{subsec:first_order_update}.  We also repeated all computation (but omitted results for conciseness) with a larger value of $\alpha$ to confirm that the resulting analysis is robust with respect to $\alpha$.

We compute hyper-differential sensitivity indices using Algorithms~\ref{alg:overview} and~\ref{alg:GHEP}. To explore the choice of eigenvalue threshold, Algorithm~\ref{alg:GHEP} is executed with $p=20$, $r_0=4$, $\Delta r =8$ (the number of processors used), and $\lambda_{min}=0.1$. This choice of $\lambda_{min}$ allows us to study the robustness of the sensitivities with respect to the eigenvalue threshold. 

\subsubsection*{Likelihood informed subspace and eigenvalue threshold}
Figure~\ref{fig:evals} displays the generalized eigenvalues (on log scale). As mentioned in Section~\ref{sec:alg_overview}, the user must specify an eigenvalue threshold $\lambda_{min}$ to define the LIS. In order to explore possible thresholds, we compute all eigenvalues greater than $0.1$ (which corresponds to the misfit magnitude being $10\%$ of the regularization magnitude in the eigenvector direction), and consider different thresholds. The vertical lines in Figure~\ref{fig:evals} indicate three possible thresholds (with the horizontal axis indicating the value of $r$ in \eqref{eqn:sen_indices_formula}). The left panel of Figure~\ref{fig:sensitivities_robust} shows the LIS-hyper-differential sensitivity indices (computed using \eqref{eqn:sen_indices_formula}) on the vertical axis as a function of the number of generalized eigenvalues (the informed subspace dimension) on the horizontal axis. We observe that the magnitude of the largest sensitivity indices varies depending on the threshold but generally the conclusions we draw from the analysis (which parameters are most/least influential) are unchanged. Zooming in on the smaller sensitivity indices in the right panel of Figure~\ref{fig:sensitivities_robust} reveals a similar conclusion for the smallest sensitivity indices. There is some variability but the conclusions of the analysis remain valid. 

\begin{figure}[h]
\centering
  \includegraphics[width=0.45\textwidth]{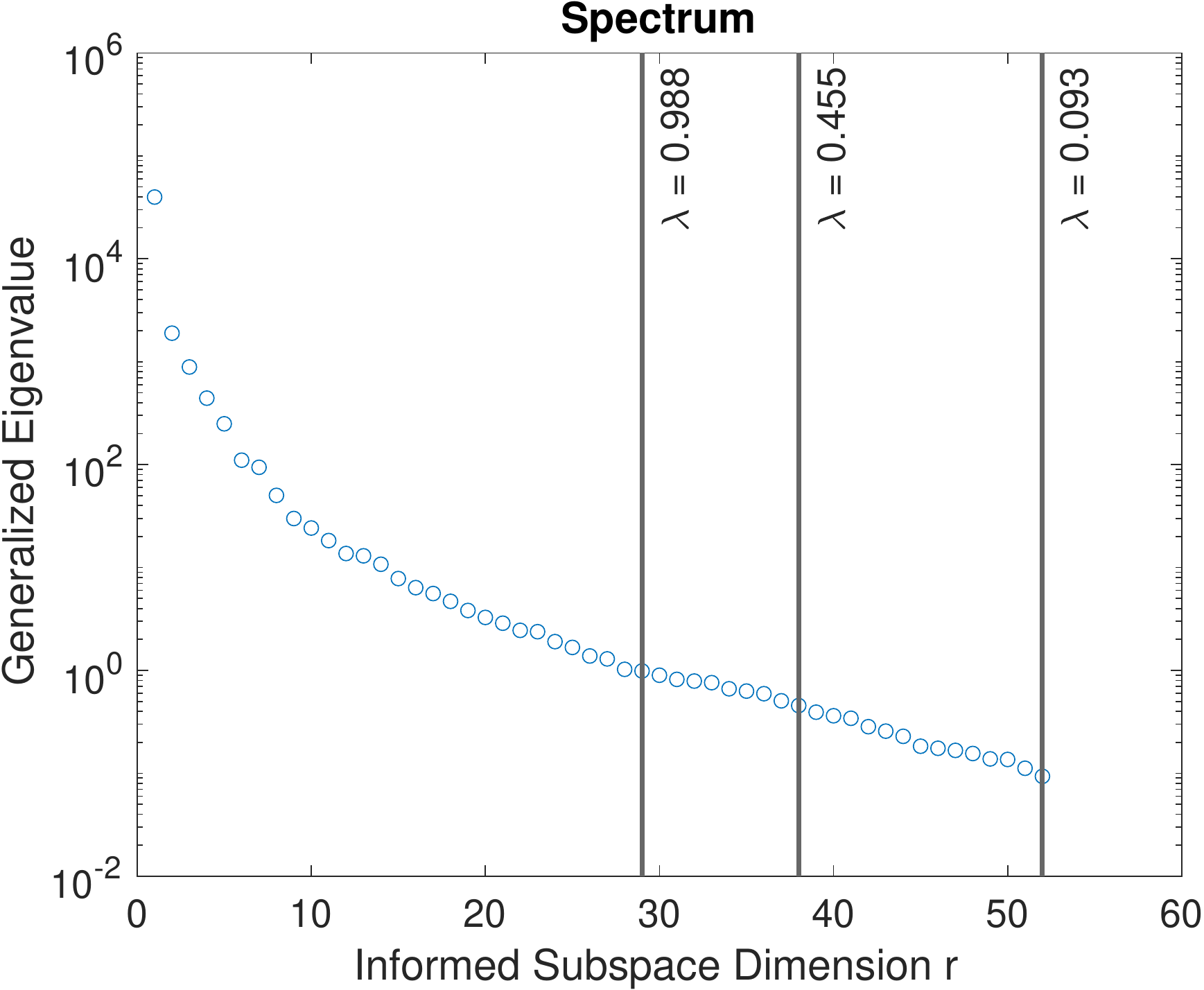}
  \caption{Log scale plot of the generalized eigenvalues. The vertical lines indicate potential eigenvalue thresholds as a frame of reference for assessing robustness.}
  \label{fig:evals}
\end{figure}

\begin{figure}[h]
\centering
    \includegraphics[width=0.45\textwidth]{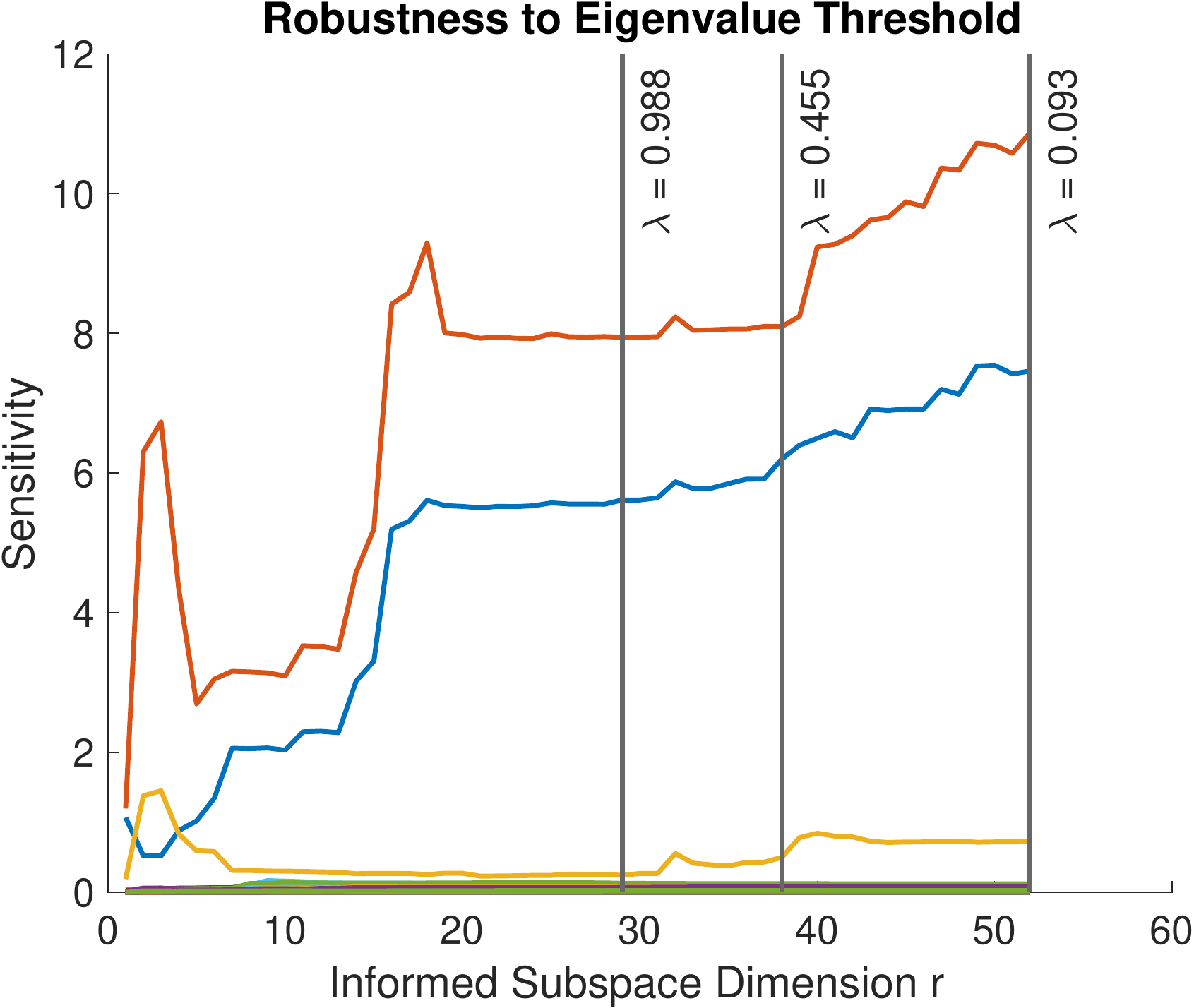}
     \includegraphics[width=0.45\textwidth]{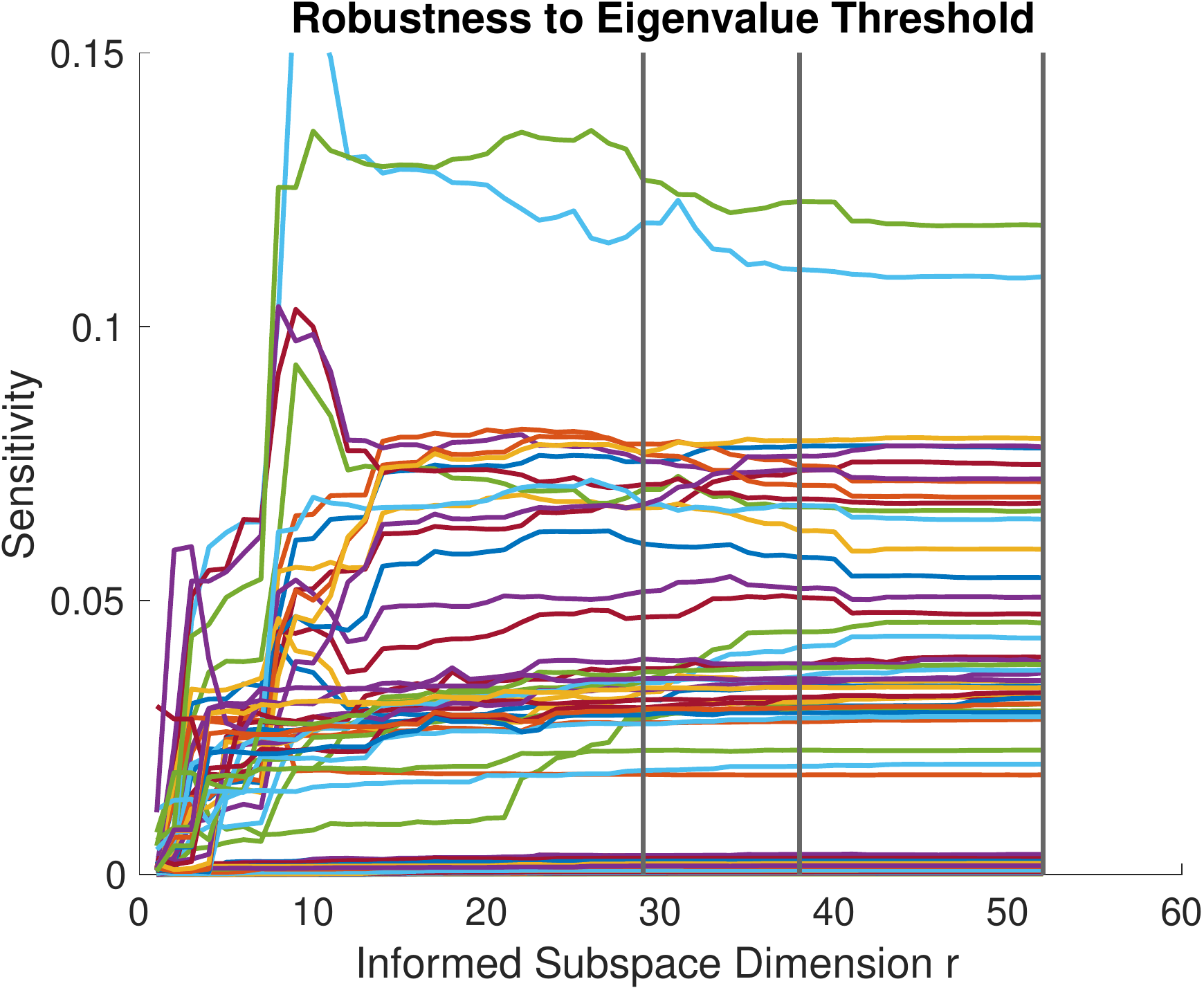}
  \caption{Left: dependence of the hyper-differential sensitivity indices (vertical axis) on the informed subspace dimension $r$ (horizontal axis); right: same plot as the left but zoomed in on the smaller sensitivity indices. The vertical lines indicate potential eigenvalue thresholds as a frame of reference for assessing robustness. The curves (which are different colors to aid visualization) correspond to different sensitivity indices.}
  \label{fig:sensitivities_robust}
\end{figure}

The spectral characteristics reveal low rank structure which is common in inverse problems. In particular, the misfit Hessian in $\mathbb R^{4225 \times 4225}$ is well characterized by a likelihood informed subspace of dimension $\mathcal O(50)$. In this example, $\vert \vert \mathcal P \vec{n} \vert \vert_{W_{\vec{z}}} = \mathcal O(10^{-4})$ and hence Theorems~\ref{thm:compare_S_and_tilde_S} and~\ref{thm:vary_alpha} ensure that the sensitivities computed with this LIS are not strongly influenced by the first order update.

\subsubsection*{Interpretation of sensitivity indices}

We take $\lambda_{min}=1$ due to its interpretation as the threshold when the misfit and regularization contributions are equal. The hyper-differential sensitivity indices are plotted in Figure~\ref{fig:sensitivities} (without the scalar sensitivity index for the diffusion coefficient which equals $0.0283$). Four different parameters (two pressure boundary conditions, a concentration source term, and diffusion coefficient) are easily compared in the HDSA framework. As discussed in subsection~\ref{ssec:solution}, the optimization problem is difficult even with these parameters fixed. Performing joint inversion on all parameters is extremely challenging and thus underscores the utility of HDSA to provide quantitative insights for complex systems with many sources of uncertainty. 

There are two particularly large sensitivity indices near $y=0.5$ on the $x=1$ pressure Dirichlet boundary. Inspection of $\mathcal B$ acting on the basis functions corresponding to these indices identifies this sensitivity as a large change in the estimated log permeability field in the $(1,0.5)$ region. This result highlights the fact that the log permeability field is not well informed by the tracer in this region (the tracer is flowing away from it) and as a result the estimated log permeability field is highly dependent on the pressure data which is strongly influence by the boundary condition near $x=1$. On the $x=0$ pressure Dirichlet boundary, we observe greater sensitivity corresponding to the lower permeability regions as the tracer outflow is determined by the combination of the permeability and low pressure regions on the boundary. The tracer injection sensitivities are generally small, indicating that errors in the tracer injection will not have a significant effect on the log permeability field estimate.
 
\begin{figure}[h]
\centering
  \includegraphics[width=0.3\textwidth]{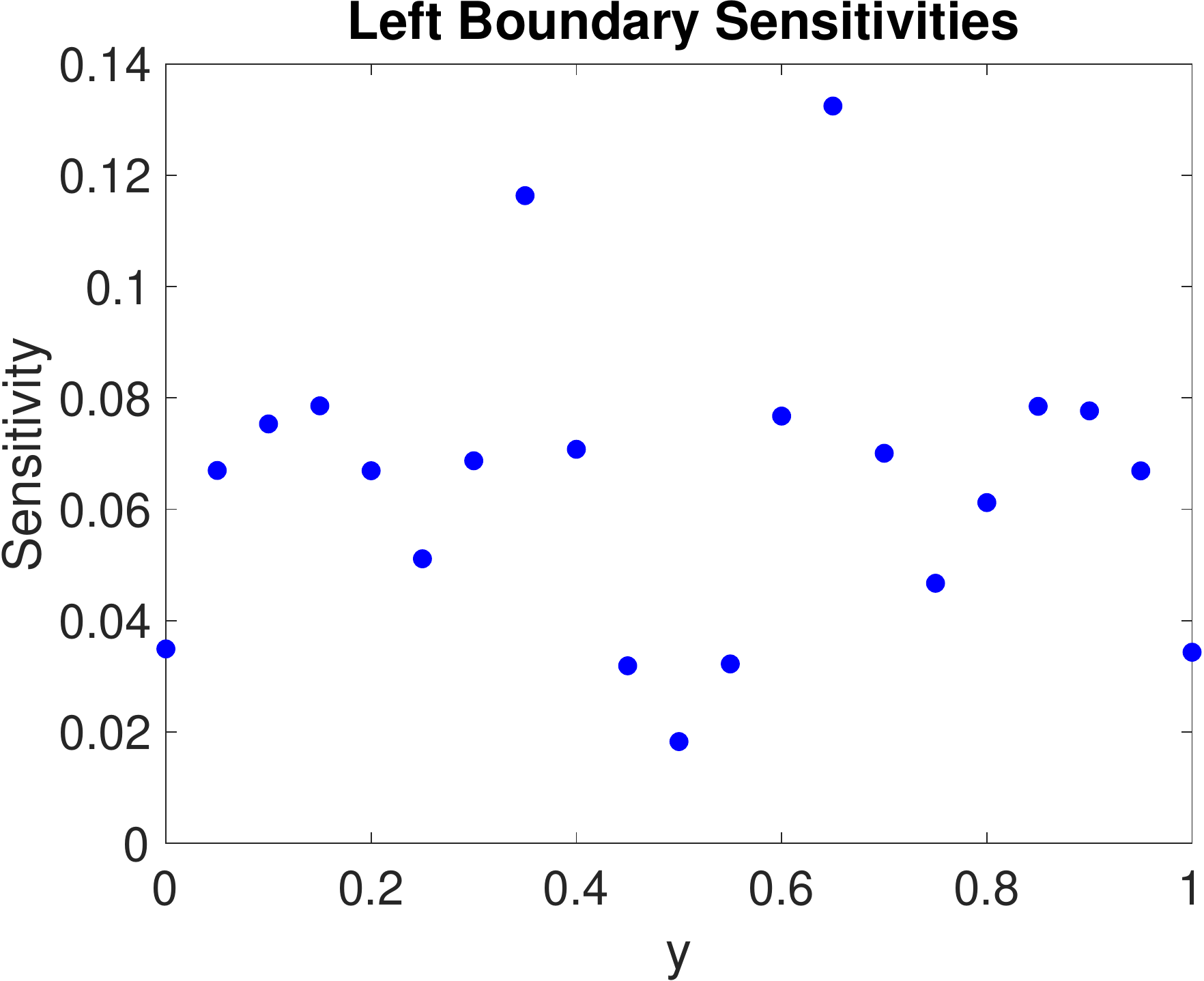}
    \includegraphics[width=0.285\textwidth]{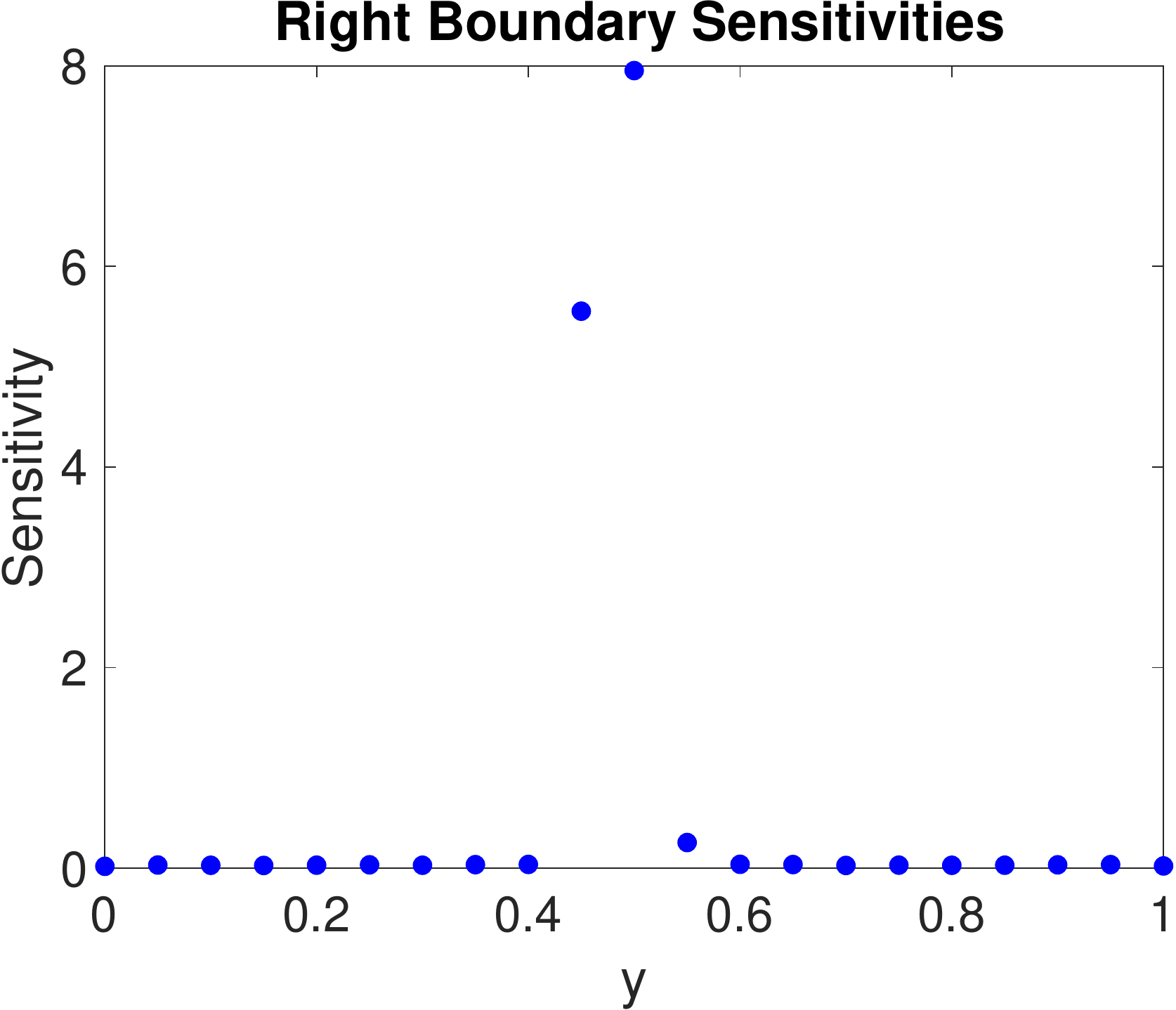}
      \includegraphics[width=0.315\textwidth]{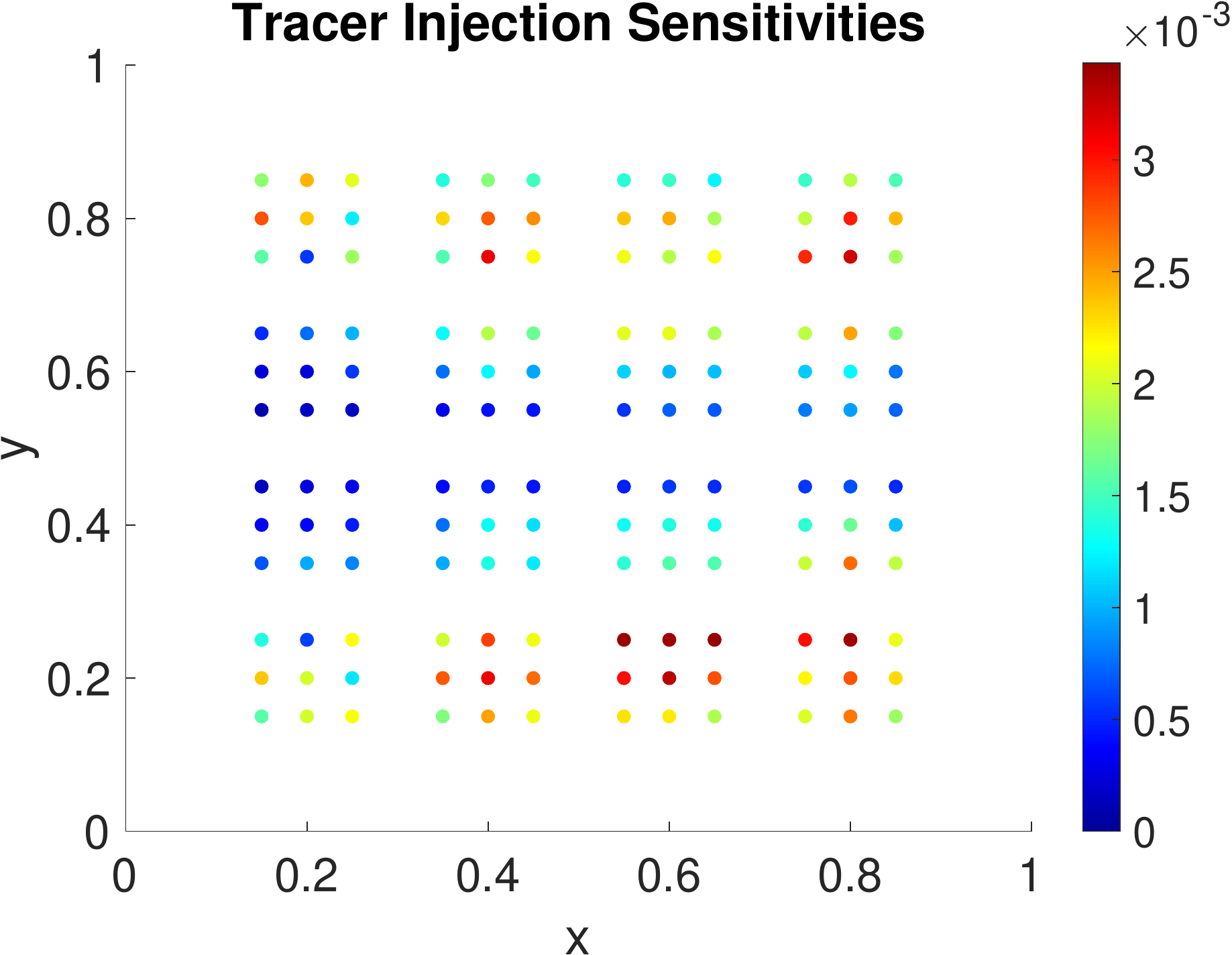}
  \caption{Hyper-differential sensitivity indices with $\lambda_{min}=1$. Left: $x=0$ pressure Dirichlet boundary condition sensitivities; center: $x=1$ pressure Dirichlet boundary condition sensitivities; right: tracer injection sensitivities.}
  \label{fig:sensitivities}
\end{figure}

The results in Figure~\ref{fig:sensitivities} may be applied for model development, experimental design, and uncertainty quantification. The influence of the pressure Dirichlet boundary conditions indicates the importance of the pressure difference driving the flow and how the log permeability field reconstruction is dependent on the features of the velocity field. The high sensitivity around $x=1$ and $y=0.5$ highlights the importance of characterizing the interaction between this region of the boundary and the log permeability field. Observing such sensitivities in practice alerts the analyst of potential concerns which may be alleviated by collecting data and/or improving modeling around the boundary. 

\section{Conclusion} \label{sec:conclusion}
Large-scale ill-posed inverse problems are challenging in many respects. Uncertainty in model parameters that are fixed (out of computational necessity) introduce uncertainty in the solution of the inverse problem which is not easily quantified in many applications. Through a likelihood informed subspace projection and development of a posteriori updates, we have enabled the use of HDSA for ill-posed inverse problems for which optimizer convergence is challenging.

Our developments provides a pragmatic approach to analyze the dependence of the optimal solution on the auxiliary model parameters. For challenging inverse problems plagued by high dimensional parameters and complex nonlinearities, HDSA provides a computationally efficient and quantitative approach to understand the complex interactions between parameters. By extending the framework to optimization problem with suboptimal solutions, we have laid a theoretical basis to use HDSA on many inverse problems arising in practice where convergence issues prohibited its use previously. Our analysis supports modeling, experimental design, and characterization of uncertainty. In particular, it may aid in understanding parameter dependences to facilitate modeling and experimental design to better characterize them. Moreover, understanding parameter dependencies is crucial to enable forward uncertainty quantification and system design under uncertainty since these dependences define a low dimensional manifold in the parameter space which must be explored in such analysis.

This article focused on optimization-based approaches to inverse problems; however, a Bayesian approach is attractive since it provides a wholistic framework for uncertainty quantification. Work is ongoing to develop and interpret HDSA for Bayesian inverse problems through its application to both the maximum a posteriori probability point and measures of posterior uncertainty~\cite{Sunseri_2022}. A key computational component of HDSA is efficient Hessian vector products. Such Hessian vector products are crucial for other aspects of large-scale inverse problems \cite{ghattas_infinite_dim_bayes_2,hierarchical_hessian_approx} and are topics of future research. There are opportunities to couple the algorithmic framework for LIS-HDSA with existing approaches for Hessian based exploration of the Bayesian posterior. 

\appendix
\section{Theorem Proofs}
\section*{Theorem~\ref{thm:sen_indices_likelihood_informed}} 
\begin{proof} 
Let $\{\lambda_j,\vec{v}_j\}_{j=1}^m$ denote the eigenvalues and eigenvectors (which are $\mathcal H_R$ orthonormal) of~\eqref{eqn:likelihood_informed_gen_eig}, i.e.\\ $\mathcal H_M \vec{v}_j = \lambda_j \mathcal H_R \vec{v}_j$. Then using the spectral representation of $ \mathcal H_M$ we have
\begin{eqnarray*}
\mathcal H = \sum\limits_{j=1}^m \lambda_j \mathcal H_R \vec{v}_j \vec{v}_j^T \mathcal H_R + \mathcal H_R
\end{eqnarray*}
and by the Sherman-Morrison-Woodbury formula we have
\begin{eqnarray*}
\mathcal H^{-1} = \mathcal H_R^{-1} - \sum\limits_{j=1}^m \frac{\lambda_j}{1+\lambda_j} \vec{v}_j \vec{v}_j^T.
\end{eqnarray*}
Noting that $\mathcal P \vec{v}_k = \vec{v}_k$ for $k=1,2,\dots,r$ and $\mathcal P \vec{v}_k =0$ for $k=r+1,r+2,\dots,m$, we have
\begin{eqnarray*}
S_i = \left\vert \left\vert \mathcal P \left( \mathcal H_R^{-1} - \sum\limits_{i=1}^m \frac{\lambda_i}{1+\lambda_i} \vec{v}_i \vec{v}_i^T \right) \mathcal B \vec{e}_i \right\vert \right\vert_{W_{\vec{z}}} =  \left\vert \left\vert  \sum\limits_{j=1}^r \frac{1}{1+\lambda_j} \vec{v}_j \left(\vec{v}_j^T \mathcal B \vec{e}_i \right)\right\vert \right\vert_{W_{\vec{z}}}.
\end{eqnarray*}
Expressing the $W_{\vec{z}}$ norm as the square root of the $W_{\vec{z}}$ inner product, i.e. $\vert \vert \cdot \vert \vert_{W_{\vec{z}}} = \sqrt{ (\cdot ,\cdot)_{W_{\vec{z}}}}$, and manipulating algebra completes the proof.
\end{proof}

\section*{Theorem~\ref{thm:optimal_update_1}}\\
We derive a closed form solution to the optimization problem (from~\eqref{R_update_gen_opt})
\begin{align}
\label{R_update_quad_opt}
& \min\limits_{\tilde{R} \in Q} \vert \vert \tilde{R} \vert \vert_{L^1(\mu)} \\
\text{s.t.} \ & \nabla_{\vec{z}} \tilde{R}(\vec{z}^\star) = -\vec{g} \nonumber
\end{align}
where $Q$ is the set of nonnegative convex quadratic functions from $\mathbb R^m$ to $\mathbb R$ and $\vec{g} = \nabla_{\vec{z}} J(\vec{z}^\star;\overline{\vec{\theta}})$. 

A general expression for a function belonging to $Q$ and satisfying $\nabla_{\vec{z}} \tilde{R}(\vec{z}^\star) = - \vec{g}$ is given by 
\begin{eqnarray}
\label{quad_rep}
\tilde{R}(\vec{z}) = \frac{1}{2}(\vec{z}-\vec{z}^\star)^T \mathcal{A} (\vec{z}-\vec{z}^\star) - \vec{z}^T \vec{g} + C
\end{eqnarray}
where $\mathcal{A} \in \mathbb R^{m \times m}$ is symmetric (since it is a Hessian) positive semidefinite (to ensure non-negativity) with $\vec{g}$ in the range of $\mathcal{A}$, and $C \in \mathbb R$ is chosen to ensure that $\tilde{R}$ is non-negative. The condition that $\vec{g}$ is in the range of $\mathcal{A}$ ensures that 
\begin{eqnarray}
\label{finding_C}
\frac{1}{2}(\vec{z}-\vec{z}^\star)^T \mathcal{A} (\vec{z}-\vec{z}^\star) - \vec{z}^T \vec{g}
\end{eqnarray}
is bounded below and hence a $C$ exists that enforces $\tilde{R} \ge 0$.  

To solve \eqref{R_update_quad_opt} we need to find a matrix $\mathcal{A}$ and a scalar $C$. Since we seek to minimize the norm of $\tilde{R}$ while enforcing $\tilde{R} \ge0$, we express $C$ as a function of $\mathcal{A}$ by setting it equal to the minimum value of \eqref{finding_C}. Setting the gradient of \eqref{finding_C} equal to zero yields the system of equations
\begin{eqnarray*}
\mathcal{A}(\vec{z}_0 - \vec{z}^\star) = \vec{g}
\end{eqnarray*}
which is guaranteed to admit a solution because $ \vec{g}$ is in the range of $\mathcal{A}$. Since, in general, $\mathcal{A}$ is positive semidefinite, the set of critical points is given by $\vec{z}^\star + \mathcal{A}^\dagger \vec{g} + N(\mathcal{A})$, where $\mathcal{A}^\dagger$ denotes the pseudo-inverse of $\mathcal{A}$ and $N(\mathcal{A})$ denotes the null space of $\mathcal{A}$. Plugging these critical points into \eqref{finding_C} yields that the smallest $C$, as a function of $\mathcal{A}$, such that $\tilde{R} \ge 0$, is given by 
\begin{eqnarray}
\label{parameterized_C}
C = \frac{1}{2} \vec{g}^T \mathcal{A}^{\dagger} \vec{g} +{\vec{z}^\star}^T \vec{g}.
\end{eqnarray}

Having written $C$ as a function of $\mathcal{A}$, we may find the optimal quadratic update by minimizing over $\mathcal{A}$. By exploiting non-negativity of $\tilde{R}$ to drop the absolute value in the $L^1$ norm, properties of the Gaussian probability measure, and ignoring terms which do not depend on $\mathcal{A}$, we arrive at the optimization problem
\begin{align} 
\label{quad_rep_opt_prob}
&\min\limits_{\mathcal{A} \in \mathcal S_m^+}  f(\mathcal{A}):=\alpha^2 Tr(\mathcal{A}) + \vec{g}^T \mathcal{A}^\dagger \vec{g} \\
& s.t. \text{ } \vec{g} \in R(\mathcal{A}) \nonumber
\end{align}
where $\mathcal S_m^+$ denotes the set of $m \times m$ positive semidefinite matrices with real entries and $Tr(\mathcal{A})$ denotes the trace of $\mathcal{A}$. 

To derive a closed form expression for~\eqref{quad_rep_opt_prob} (and equivalently \eqref{R_update_quad_opt}), we prove a sequence of lemmas to arrive at a unique solution. Using the spectral theorem for symmetric matrices we may, without loss of generality, assume that the optimal solution of \eqref{quad_rep_opt_prob} is a matrix $\mathcal{H}_{\tilde{R}}$ of the form
\begin{eqnarray}
\label{spec_rep}
\mathcal{H}_{\tilde{R}}  = \sum\limits_{j=1}^m \lambda_j \vec{v}_j \vec{v}_j^T
\end{eqnarray}
where $\lambda_1 \ge \lambda_2 \ge \cdots \ge \lambda_m \ge 0$ are the non-negative eigenvalues and $\{\vec{v}_j\}_{j=1}^m$ are the orthonormal eigenvectors of $\mathcal H_{\tilde{R}}$.

Let $r \in \{1,2,\dots,m\}$ be the rank of $\mathcal{H}_{\tilde{R}}$. Plugging \eqref{spec_rep} into \eqref{quad_rep_opt_prob}, writing the trace as the sum of eigenvalues, and using the spectral decomposition to express the pseudo-inverse, we find that
\begin{eqnarray*}
f(\mathcal{H}_{\tilde{R}})=\alpha^2 \sum\limits_{j=1}^r \lambda_j + \sum\limits_{j=1}^r \frac{1}{\lambda_j} (\vec{g}^T\vec{v}_j)^2 . 
\end{eqnarray*}
The subsequent lemmas characterize the eigenvalues $\{\lambda_j\}_{j=1}^r$ and eigenvectors $\{\vec{v}_j\}_{j=1}^r$ by utilizing the fact that $f(\mathcal{H}_{\tilde{R}}) \le f(\mathcal{A})$ for all $\mathcal{A} \in \mathcal S_m^+$ which satisfy $\vec{g} \in R(\mathcal{A})$.

Let $\lambda_1$ have algebraic multiplicity $\ell$, i.e. $\lambda_1=\lambda_2=\cdots=\lambda_\ell > \lambda_{\ell+1}$ (or $\lambda_j=\lambda_1$, $j=1,2,\dots,m$ if $\ell=m$). Then we have the following.
\begin{lemma}
\label{lemma:leading_eval}
\begin{eqnarray*}
\vec{g} \in \text{span}\{\vec{v}_1,\vec{v}_2,\dots,\vec{v}_\ell\}.
\end{eqnarray*}
\end{lemma}
\begin{proof}
We may rewrite the objective function as
\begin{eqnarray*}
\label{eq:prop1}
f(\mathcal{H}_{\tilde{R}})=\alpha^2 \sum\limits_{j=1}^r \lambda_j + \vert \vert \vec{g} \vert \vert_2^2 \sum\limits_{j=1}^r \gamma_j \frac{1}{\lambda_j}
\end{eqnarray*}
where $\gamma_j= (\vec{g}^T\vec{v}_j)^2/\vert \vert \vec{g} \vert \vert_2^2 \in [0,1]$ and $\sum_{i=1}^r \gamma_i=1$ by Parseval's identity. Because of the ordering of the eigenvalues, $f(\mathcal{H}_{\tilde{R}})$ is minimized when $\gamma_i=0$ for $i>\ell$. Hence, $(\vec{g}^T\vec{v}_j)=0$ for $j>\ell$. Since $\{\vec{v}_j\}_{j=1}^m$ is an orthonormal basis for $\mathbb R^m$, 
\begin{eqnarray*}
\vec{g} = \sum\limits_{j=1}^m (\vec{g}^T\vec{v}_j)\vec{v}_j = \sum\limits_{j=1}^\ell (\vec{g}^T\vec{v}_j)\vec{v}_j .
\end{eqnarray*}
\hspace{1 mm}
\end{proof}

Applying Lemma~\ref{lemma:leading_eval}, we can simplify $f(\mathcal{H}_{\tilde{R}})$ to get
\begin{eqnarray}
\label{obj_reduced_1}
f(\mathcal{H}_{\tilde{R}})=\alpha^2 \sum\limits_{j=1}^r \lambda_j + \vert \vert \vec{g} \vert \vert_2^2 \frac{1}{\lambda_1} . 
\end{eqnarray}
\begin{lemma}
\label{lemma:leading_eval2}
$\lambda_j=0$ for $j>\ell$.
\end{lemma}
\begin{proof}
The result follows by observing that all quantities in~\eqref{obj_reduced_1} are non-negative and it is minimized when there are a minimum number of terms in the sum of eigenvalues.
\end{proof}

We may now rewrite $f(\mathcal{H}_{\tilde{R}})$ as
\begin{eqnarray}
\label{obj_reduced_2}
f(\mathcal{H}_{\tilde{R}})=\alpha^2 \ell \lambda_1 + \vert \vert \vec{g} \vert \vert_2^2 \frac{1}{\lambda_1}.
\end{eqnarray}

\begin{lemma}
\label{lemma:leading_eval3}
$\ell=1$ and $\vec{v}_1 = \pm \frac{\vec{g}}{ \vert \vert \vec{g} \vert \vert_2}$.
\end{lemma}
\begin{proof}
Since all quantities in \eqref{obj_reduced_2} are nonnegative and $\ell \in \{1,2,\dots,m\}$, then \eqref{obj_reduced_2} will be minimized when $\ell$ is minimized. Taking $\ell=1$ clearly minimizes it and also implies that $\vec{v}_1 = \pm \frac{\vec{g}}{ \vert \vert \vec{g} \vert \vert_2}$ to ensure that $\vec{g} \in R(\mathcal{H}_{\tilde{R}})$.
\end{proof}

Hence, we have established that $\mathcal{H}_{\tilde{R}}$ is given by the rank one matrix
\begin{eqnarray}
\label{opt_A_1}
\mathcal{H}_{\tilde{R}} = \frac{\lambda}{\vert \vert \vec{g} \vert \vert_2^2} \vec{g} \vec{g}^T
\end{eqnarray}
where $\lambda$ is the optimal eigenvalue which minimizes
\begin{eqnarray}
\label{obj_reduced_3}
\alpha^2 \lambda + \vert \vert \vec{g} \vert \vert_2^2 \frac{1}{\lambda}.
\end{eqnarray}
By setting the derivative of \eqref{obj_reduced_3} with respect to $\lambda$ equal to zero and solving, we find that $\mathcal{H}_{\tilde{R}}$ is given by
\begin{eqnarray}
\label{opt_A_2}
\mathcal{H}_{\tilde{R}} = \frac{1}{\alpha \vert \vert \vec{g} \vert \vert_2  } \vec{g} \vec{g}^T.
\end{eqnarray}
Hence, the solution of \eqref{R_update_quad_opt} is \eqref{quad_rep} with $\mathcal{A}$ given by \eqref{opt_A_2} and $C$ given by \eqref{parameterized_C}.

\section*{Theorem~\ref{thm:compare_S_and_tilde_S}}
\begin{proof}
First observe that $\tilde{\mathcal B}=\mathcal B$ since $\tilde{J}=J+\tilde{R}$ and $\tilde{R}$ does not depend on $\vec{\theta}$. Hence, 
\begin{eqnarray*}
\tilde{S}_i = \vert \vert \mathcal P (\mathcal H+\mathcal{H}_{\tilde{R}})^{-1} \mathcal B \vec{e}_i \vert \vert_{W_{\vec{z}}}  \qquad \text{ and } \qquad S_i =  \vert \vert \mathcal P \mathcal H^{-1} \mathcal B \vec{e}_i \vert \vert_{W_{\vec{z}}},
\end{eqnarray*}
where $\mathcal{H}_{\tilde{R}} = \vec{g} \vec{g}^T / (\alpha \vert \vert \vec{g} \vert \vert_2)$ is the Hessian of $\tilde{R}$.

Applying the reverse triangle inequality yields
\begin{align*}
\vert \tilde{S}_i - S_i \vert \le \vert \vert \mathcal P \left( (\mathcal H+\mathcal{H}_{\tilde{R}})^{-1} \mathcal B \vec{e}_i -\mathcal H^{-1} \mathcal B \vec{e}_i \right) \vert \vert_{W_{\vec{z}}} .
\end{align*}
Using the Sherman-Morrison formula for rank one updates and properties of norms we have
\begin{align*}
\vert \tilde{S}_i - S_i \vert & \le \left\vert \left\vert \mathcal P  \frac{1}{1+\frac{\vert \vert \vec{g}\vert \vert_2}{\alpha} \vec{s}^T \mathcal H^{-1} \vec{s} } \mathcal H^{-1}\mathcal{H}_{\tilde{R}}\mathcal H^{-1} \mathcal B \vec{e}_i \right\vert \right\vert_{W_{\vec{z}}} \\
& \le    \frac{1}{1+\frac{\vert \vert \vec{g} \vert \vert_2}{\alpha}\vec{s}^T \mathcal H^{-1} \vec{s} }  \cdot \left\vert \left\vert \mathcal P  \mathcal H^{-1}\mathcal{H}_{\tilde{R}} \right\vert \right\vert_{W_{\vec{z}}} \cdot \left\vert \left\vert  \mathcal H^{-1} \mathcal B \vec{e}_i \right\vert \right\vert_{W_{\vec{z}}}
\end{align*}
Dividing by $\left\vert \left\vert \mathcal H^{-1} \mathcal B \vec{e}_i \right\vert \right\vert_{W_{\vec{z}}}$, plugging in $\mathcal{H}_{\tilde{R}} = \vec{g} \vec{g}^T / (\alpha \vert \vert \vec{g} \vert \vert_2)$, and manipulating constants yields
\begin{align}
\label{bound_1}
\frac{\vert \tilde{S}_i - S_i \vert}{\left\vert \left\vert \mathcal H^{-1} \mathcal B \vec{e}_i \right\vert \right\vert_{W_{\vec{z}}}} & \le \frac{ \vert \vert \vec{g} \vert \vert_2 }{ \alpha + \vec{s}^T \vec{n}} \vert \vert \mathcal P  \mathcal H^{-1} \vec{s} \vec{s}^T \vert \vert_{W_{\vec{z}}} .
\end{align}
Notice that $\vert \vert \mathcal P \mathcal H^{-1} \vec{s} \vec{s}^T \vert \vert_{W_{\vec{z}}} = \vert \vert \mathcal P \mathcal H^{-1} \vec{s} \vert \vert_{W_{\vec{z}}}$ since it is a rank one matrix given by the outer product of $\mathcal P \mathcal H^{-1} \vec{s}$ and $\vec{s}$ with $\vert \vert \vec{s} \vert \vert_2 =1$. Plugging this into \eqref{bound_1} and manipulating constants completes the proof.
\end{proof}

\section*{Theorem~\ref{thm:vary_alpha}}
\begin{proof}
Noting that $\tilde{\mathcal B}(\alpha) \vec{e}_i = \tilde{\mathcal B}(\alpha+\alpha \beta)\vec{e}_i = \mathcal B \vec{e}_i$ for any $\beta \in (-1,1)$, applying the reverse triangle inequality, and writing the Hessian of $\tilde{R}$ as $\mathcal{H}_{\tilde{R}}(\alpha)$ yields
\begin{align*}
\vert \tilde{S}_i(\alpha+\alpha \beta) - \tilde{S}_i(\alpha) \vert \le \left\vert \left\vert \mathcal P \left( \mathcal H + \mathcal{H}_{\tilde{R}}(\alpha (1+\beta)) \right)^{-1} \mathcal B \vec{e}_i - \mathcal P \left( \mathcal H + \mathcal{H}_{\tilde{R}}(\alpha) \right)^{-1} \mathcal B \vec{e}_i \right\vert \right\vert_{W_{\vec{z}}} .
\end{align*}
Following the same arguments as in the proof of Theorem~\ref{thm:compare_S_and_tilde_S} with the Sherman-Morrison formula for rank one updates and properties of norms we have
\begin{align*}
\frac{\vert \tilde{S}_i(\alpha+\alpha \beta) - \tilde{S}_i(\alpha) \vert}{\left\vert \left\vert \mathcal H^{-1} \mathcal B \vec{e}_i \right\vert \right\vert_{W_{\vec{z}}}} & \le \left\vert \frac{1}{\alpha+ \vec{s}^T\vec{n}} - \frac{1}{\alpha(1+\beta) + \vec{s}^T\vec{n}} \right\vert \cdot \frac{1}{\vert \vert \vec{g} \vert \vert_2} \cdot \vert \vert \mathcal P \mathcal H^{-1} \vec{g} \vec{g}^T \vert \vert_{W_{\vec{z}}}
\end{align*}
Then $\vert \vert \mathcal P \mathcal H^{-1} \vec{g} \vec{g}^T \vert \vert_{W_{\vec{z}}} = \vert \vert \vec{g} \vert \vert_2 \cdot \vert \vert \mathcal P \vec{n} \vert \vert_{W_{\vec{z}}}$ since $\vert \vert  \mathcal P \mathcal H^{-1} \vec{s} \vec{s}^T \vert \vert_{W_{\vec{z}}} = \vert \vert \mathcal P  \mathcal H^{-1} \vec{s} \vert \vert_{W_{\vec{z}}}$. 

Hence,
\begin{align*}
\frac{\vert \tilde{S}_i(\alpha+\alpha \beta) - \tilde{S}_i(\alpha) \vert}{\left\vert \left\vert \mathcal H^{-1} \mathcal B \vec{e}_i \right\vert \right\vert_{W_{\vec{z}}}} & \le \left\vert \frac{1}{\alpha+ \vec{s}^T\vec{n}} - \frac{1}{\alpha(1+\beta) + \vec{s}^T\vec{n}} \right\vert \vert \vert \mathcal P \vec{n} \vert \vert_{W_{\vec{z}}} .
\end{align*}
Algebraic manipulations give
\begin{align*}
\frac{\vert \tilde{S}_i(\alpha+\alpha \beta) - \tilde{S}_i(\alpha) \vert}{\left\vert \left\vert \mathcal H^{-1} \mathcal B \vec{e}_i \right\vert \right\vert_{W_{\vec{z}}}} & \le \vert \beta \vert \cdot \frac{\alpha \vert \vert \mathcal P \vec{n} \vert \vert_{W_{\vec{z}}}}{(\alpha+\vec{s}^T\vec{n)}(\alpha(1+\beta)+\vec{s}^T\vec{n})}.
\end{align*}
Noting that $\vec{s}^T\vec{n}>0$ so $\alpha / (\alpha + \vec{s}^T\vec{n}) <1$ completes the proof.
\end{proof}

\section*{Theorem~\ref{thm:eig_compare}}
\begin{proof}
$(\rightarrow)$ Assume that $\mathcal H_M+\mathcal H_R$ is positive definite and let $\mathcal H_M \vec{v}_j = \lambda_j \mathcal H_R \vec{v}_j$, $j=1,2,\dots,m$. Without loss of generality, assume that $\vec{v}_j^T \mathcal H_R \vec{v}_j=1$, $j=1,2,\dots,m$. Then
$$0 < \vec{v}_j^T (\mathcal H_M+\mathcal H_R)\vec{v}_j = \vec{v}_j^T(\lambda_j \mathcal H_R\vec{v}_j + \mathcal H_R\vec{v}_j) = (\lambda_j+1)\vec{v}_j^T \mathcal H_R \vec{v}_j = \lambda_j+1.$$ 
$(\leftarrow)$ Assume that $\mathcal H_M \vec{v}_j = \lambda_j \mathcal H_R \vec{v}_j$ with $\lambda_j > -1$, $j=1,2,\dots,m$. Let $\vec{x} \in \mathbb R^m, \vec{x} \ne 0$. It is sufficient to show that $\vec{x}^T(\mathcal H_M+\mathcal H_R)\vec{x} > 0$. Since $\{\vec{v}_j\}_{j=1}^m$ forms a $\mathcal H_R$ orthonormal basis for $\mathbb R^m$ we can write $\vec{x} = \sum_{j=1}^m (\vec{x}^T \mathcal H_R \vec{v}_j) \vec{v}_j.$
Then we have
\begin{align*}
\vec{x}^T(\mathcal H_M+\mathcal H_R)\vec{x} & = \left( \sum\limits_{j=1}^m (\vec{x}^T \mathcal H_R \vec{v}_j) \vec{v}_j \right)^T (\mathcal H_M+\mathcal H_R) \left(\sum\limits_{k=1}^m (\vec{x}^T \mathcal H_R \vec{v}_k) \vec{v}_k \right) \\
& =  \left( \sum\limits_{j=1}^m (\vec{x}^T \mathcal H_R \vec{v}_j) \vec{v}_j \right)^T  \left( \sum\limits_{k=1}^m (\vec{x}^T \mathcal H_R \vec{v}_k) (\mathcal H_M\vec{v}_k + \mathcal H_R\vec{v}_k) \right) \\
& =  \left( \sum\limits_{j=1}^m (\vec{x}^T \mathcal H_R \vec{v}_j) \vec{v}_j \right)^T  \left( \sum\limits_{k=1}^m (\vec{x}^T \mathcal H_R \vec{v}_k) (\lambda_k+1)\mathcal H_R\vec{v}_k \right) \\
& =  \sum\limits_{j=1}^m (\vec{x}^T \mathcal H_R \vec{v}_j)^2 (\lambda_j+1) \\
& > 0
\end{align*}
since $\lambda_j+1 > 0$ for all $j=1,2,\dots,m$.
\end{proof}

\section*{Theorem~\ref{thm:gen_reg_update}}
\begin{proof}
For $j\ne i$ we have,
\begin{eqnarray*}
(\mathcal H_M+\mathcal U_i(\delta))\vec{v}_j = \mathcal H_M\vec{v}_j + \delta \mathcal H_R \vec{v}_i \vec{v}_i^T \mathcal H_R \vec{v}_j = \lambda_j \mathcal H_R \vec{v}_j + \delta \mathcal H_R \vec{v}_i (0) =  \lambda_j \mathcal H_R \vec{v}_j
\end{eqnarray*}
since $\{\vec{v}_j\}_{j=1}^m$ are $\mathcal H_R$ orthogonal (a consequence of the spectral theorem for symmetric matrices). As for $\vec{v}_i$, consider 
\begin{eqnarray*}
(\mathcal H_M+\mathcal U_i(\delta))\vec{v}_i = \mathcal H_M\vec{v}_i + \delta \mathcal H_R \vec{v}_i \vec{v}_i^T \mathcal H_R \vec{v}_i = \lambda_i \mathcal H_R \vec{v}_i + \delta \mathcal H_R \vec{v}_i (1) =  (\lambda_m+\delta) \mathcal H_R \vec{v}_i.
\end{eqnarray*}
\hspace{1 mm}
\end{proof}

\bibliographystyle{siam}
\bibliography{dasco}

\end{document}